\documentclass[12pt,reqno]{amsart}  
\usepackage{amsmath,amsthm,amssymb,amscd,mathdots,mathtools,enumerate,shuffle,stmaryrd,faktor,hyperref,xcolor,float,scalerel,stackengine,bm}
\allowdisplaybreaks

\newtheorem{thm}{Theorem}[section]
\newtheorem{prop}[thm]{Proposition}
\newtheorem{lem}[thm]{Lemma}
\newtheorem{cor}[thm]{Corollary}
\newtheorem{con}[thm]{Conjecture}
\numberwithin{equation}{section}

\newtheorem*{thm*}{Theorem}

\theoremstyle{definition}
\newtheorem{defi}[thm]{Definition}

\newtheorem{rem}[thm]{Remark}
\newcommand{\medtildeM}{ \ThisStyle{\stackon[+0.9pt]{$\mathrm{M}$}{ \stretchto{\scaleto{\mkern0.2mu\bm\sim}{3.3\LMpt} }{2\LMpt} }}}

\newcommand{\ZZ}{\mathbb{Z}}
\newcommand{\QQ}{\mathbb{Q}}
\newcommand{\quasimod}{\ensuremath{\medtildeM\!(\operatorname{Sl}_2(\ZZ))}}
\newcommand{\quasimodwt}[1]{\ensuremath{\medtildeM_{#1}(\operatorname{Sl}_2(\ZZ))}}
\newcommand{\cusp}[1]{S_{#1}(\operatorname{Sl}_2(\ZZ))}

\newcommand{\Z}{\mathcal{Z}}
\newcommand{\grDZ}{\operatorname{gr}_D\Z}
\newcommand{\grD}{\operatorname{gr}_D}
\newcommand{\fil}{\operatorname{Fil}}

\newcommand{\BARI}{\operatorname{BARI}_{\underline{\operatorname{al}},\operatorname{swap}}}
\newcommand{\bimap}{\boldsymbol{\beta}}

\newcommand{\wt}{\operatorname{wt}}
\newcommand{\dep}{\operatorname{dep}}
\newcommand{\ad}{\operatorname{ad}}
\newcommand{\id}{\operatorname{id}}

\newcommand{\B}{\mathcal{B}}
\newcommand{\QB}{\QQ\langle \B\rangle}
\newcommand{\RB}{R\langle\langle \B\rangle\rangle}
\newcommand{\LieB}{\operatorname{Lie}\langle \B\rangle}

\newcommand{\X}{\mathcal{X}}
\newcommand{\QX}{\QQ\langle \X\rangle}

\newcommand{\Y}{\mathcal{Y}}
\newcommand{\QY}{\QQ\langle \Y\rangle}

\newcommand{\Dbi}{\mathcal{D}_{\operatorname{bi}}}
\newcommand{\QDbi}{\QQ\langle \Dbi\rangle}
\newcommand{\LieDbi}{\operatorname{Lie}\langle \Dbi\rangle}

\newcommand{\ls}{\mathfrak{ls}}
\newcommand{\dm}{\mathfrak{dm}_0}
\newcommand{\lqq}{\mathfrak{lq}}
\newcommand{\bmzero}{\mathfrak{bm}_0}

\newcommand{\ariemp}{\{-,-\}\!_A}
\newcommand{\ari}[2]{\{#1,#2\}\!_A}
\newcommand{\arizero}[2]{\{#1,#2\}\!_{A,0}}
\newcommand{\der}[1]{d_{#1}^A}
\newcommand{\derr}[1]{d_{#1}^{A,r}}
\newcommand{\derl}[1]{d_{#1}^{A,l}}
\newcommand{\derzero}[1]{d_{#1}^{A,0}}
\newcommand{\derrzero}[1]{d_{#1}^{A,r,0}}
\newcommand{\derlzero}[1]{d_{#1}^{A,l,0}}
\newcommand{\derbull}[1]{d_{#1}^{A,\bullet}}
\newcommand{\derbullzero}[1]{d_{#1}^{A,\bullet,0}}

\newcommand{\conczero}{\cdot_0}
\newcommand{\shco}{\Delta_{\shuffle}}

\newcommand{\pizero}{\Pi_0}
\newcommand{\piY}{\Pi_\Y}
\newcommand{\thetaX}{\theta_\X}
\newcommand{\thetaY}{\theta_\Y}

\title{An extension of the linearized double shuffle Lie algebra}
\author{Annika Burmester}
\address{Graduate School of Mathematics, Nagoya University, Japan.}
\email{annika.burmester@math.nagoya-u.ac.jp}

\begin{document}
	\maketitle
\begin{abstract}
The linearized double shuffle Lie algebra $\ls$ is a well-studied Lie algebra, which reflects the depth-graded structure of multiple zeta values. We introduce a generalization $\lqq$, which is motivated from the $\mathbb{Q}$-algebraic structure of multiple q-zeta values and multiple Eisenstein series. Precisely, we show that $\lqq$ is a Lie algebra, where the Lie bracket is related to Ecalle's ari bracket on bimoulds, and give an embedding of $\ls$ into $\lqq$.
\end{abstract}

\section{Introduction}
	
\textbf{1. Multiple zeta values.} For $k_1,\ldots,k_d\in \ZZ_{\geq1},\ k_1\geq2$, a \emph{multiple zeta value} of \emph{weight} $k_1+\cdots+k_d$ and depth $d$ is
\begin{align*}
\zeta(k_1,\ldots,k_d)=\sum_{n_1>\cdots>n_d>0} \frac{1}{n_1^{k_1}\cdots n_d^{k_d}} \in \mathbb{R}.
\end{align*}
In recent decades, they have attracted much interest due to their appearance in various fields of mathematics and in physics, like number theory, algebraic geometry, knot theory, quantum field theory, as well as for their rich algebraic structure. One key feature of multiple zeta values is their two product expressions, the \emph{stuffle product} coming from the combinatorics of nested sums and the \emph{shuffle product} coming from their iterated integral expression. For example, we have for $k_1,k_2\geq2$
\begin{align*}
\zeta(k_1)\zeta(k_2)&=\zeta(k_1,k_2)+\zeta(k_2,k_1)+\zeta(k_1+k_2) \\
&=\sum_{j=2}^{k_1+k_2}\left(\binom{j-1}{k_1-1}+\binom{j-1}{k_2-1}\right) \zeta(j,k_1+k_2-j).
\end{align*}
The comparison of the two products along with some regularization leads to the \emph{extended double shuffle relations}. Ihara, Kaneko and Zagier conjectured in \cite{IKZ06} that these relations generate all the numerous algebraic relations among multiple zeta values.
One way to study the extended double shuffle relations is to pass to a linearized, depth-graded setting \cite[Section 8]{IKZ06}. Rephrasing this idea in a more algebraic way leads to the \emph{linearized double shuffle Lie algebra} $\ls$ studied by Brown \cite{Bro21}. Roughly, $\ls$ consists of non-commutative polynomials $\psi$ in two letters $x_0,x_1$, such that
\begin{itemize}
\item[(i)] $\psi$ does not have weight $1$ terms or depth $1$ terms of even weight,
\item[(ii)] $\psi$ is primitive for the shuffle coproduct,
\item[(iii)] if we replace each $x_0^{n-1}x_1$ by $y_n$ and words ending in $x_0$ by $0$ in $\psi$, the result is primitive for the depth-graded stuffle coproduct.
\end{itemize}
The Lie bracket on $\ls$ is given by the \emph{Ihara bracket} $\{-,-\}$. By construction and \cite[Conjecture 1]{IKZ06}, we expect for the depth-graded algebra $\grDZ$ of multiple zeta values that there is an algebra isomorphism
\begin{align} \label{eq:iso_grZ_Uls}
\grDZ/(\zeta(2))\overset{?}{\simeq} \mathcal{U}(\ls)^\vee,
\end{align} 
where $(-)^\vee$ denotes the graded dual with respect to the weight.
The Lie algebra $\ls$ is more accessible than the extended double shuffle relations. For instance, computing the dimensions of $\ls$ provides an upper bound for the weight- and depth-homogeneous components of the algebra $\Z$ of multiple zeta values. Passing to the associated depth-graded reveals a connection of multiple zeta values and cusp forms for $\operatorname{Sl}_2(\ZZ)$, as studied in \cite{BK97}, \cite{Bro21}.

\vspace{0,3cm}
\textbf{2. Variations of multiple zeta values and Eisenstein series.} An approach to understanding the algebra of multiple zeta values and the extended double shuffle relations is to consider variants of them, where some structures become simpler. \\
\emph{q-analogs} of multiple zeta values are one such approach, those are $q$-series degenerating to multiple zeta values under the limit $q\to1$. There are lots of different models of them in the literature, one of them are the \emph{Schlesinger-Zudilin multiple $q$-zeta values}
\begin{align*}
\zeta_q^{\operatorname{SZ}}(s_1,\ldots,s_l)=\sum_{0<n_1<\cdots< n_l} \frac{q^{n_1s_1}}{(1-q^{n_1})^{s_1}}\cdots \frac{q^{n_ls_l}}{(1-q^{n_l})^{s_l}} \in \QQ[[q]],
\end{align*} 
where $s_1,\ldots,s_l\in \ZZ_{\geq0},\ s_l\geq1$ and $q$ is just a formal variable. We have reversed the order of the summation indices compared to multiple zeta values to make the corresponding Lie algebra structure introduced below more naturally compatible with $\ls$. In \cite{BK20} it is shown that many models for multiple $q$-zeta values are contained in the space $\Z_q$ spanned by the Schlesinger-Zudilin multiple $q$-zeta values, so we may study all of them simultaneously. One big advantage of $q$-analogs of multiple zeta values is that we deal with $q$-series instead of real numbers. This allows for addressing questions about linear independence, as for example in \cite{BK20}, which are currently inaccessible for multiple zeta values. From our point of view the main advantage of multiple $q$-zeta values is that instead of two product expressions together with regularization (i.e., the extended double shuffle relations), it conjecturally suffices to consider the usual stuffle product together with the invariance under some involution $\tau$, which reads explicitly for $m_i\geq0,\ k_i\geq1$
\begin{align*}
\zeta_q^{\operatorname{SZ}}(\{0\}^{m_1},k_1,\ldots,\{0\}^{m_d},k_d)=\zeta_q^{\operatorname{SZ}}(\{0\}^{k_d-1},m_d+1,\ldots,\{0\}^{k_1-1},m_1+1).
\end{align*}
\emph{Multiple Eisenstein series}, which were introduced in \cite{GKZ06}, have a deep connection to the space $\Z_q$. They are defined for $k_1,\ldots,k_d\in \ZZ_{\geq2}, k_d\geq 3,$ by 
\[\mathbb{G}_{k_1,\ldots,k_d}(\tau)=\sum_{\substack{0\prec \lambda_1\prec \cdots \prec \lambda_d \\ \lambda_i \in \ZZ\tau+\ZZ}} \frac{1}{\lambda_1^{k_1}\cdots \lambda_d^{k_d}} \qquad (\tau\in \mathbb{H}).\]
The ordering on $\ZZ\tau+\ZZ$ is understood as $m_1\tau+n_1 \prec m_2\tau+n_2$ if and only if $m_1< m_2$ or $m_1=m_2$ and $n_1<n_2$. The Fourier expansion of multiple Eisenstein series was computed in \cite{Ba12}, the constant term of $\mathbb{G}_{k_1,\ldots,k_d}(\tau)$ is exactly $\zeta(k_1,\ldots,k_d)$ and also the other Fourier coefficients are $\QQ$-linear combinations of multiple zeta values.
\\ 
The \emph{combinatorial bi-multiple Eisenstein series}  $G\binom{k_1,\ldots,k_d}{m_1,\ldots,m_d}$ from \cite{BB23} span the space $\Z_q$ and their single indexed versions $G(k_1,\ldots,k_d)$ conjecturally satisfy the same $\QQ$-algebraic relations as multiple Eisenstein series. Thus, we view them as a rational $q$-series counterpart of multiple Eisenstein series. Since the construction of the combinatorial bi-multiple Eisenstein series depends on a rational solution to the extended double shuffle equations, they are not unique and cannot be made fully explicit. The extension to bi-indices encounters partial derivatives with respect to $q\frac{d}{dq}$, and hence allows to describe all expected relations among combinatorial bi-multiple Eisenstein series with the usual (bi-)stuffle product and the invariance under some involution swap. Conjecturally (\cite[Remark 6.11]{BB23}), already the combinatorial multiple Eisenstein series $G(k_1,\ldots,k_d)$ suffice to span the space $\Z_q$. This point of view gives another motivation for a deeper understanding of the algebraic structure of $\Z_q$. In particular, $\Z_q$ contains all quasi-modular forms and can be seen as an extension of them (see e.g., \cite{BI24}).

\vspace{0,3cm}
\textbf{A bi-graded Lie algebra for $\Z_q$.}  Inspired by the linearized double shuffle Lie algebra $\ls$ for multiple zeta values, we study the relations in the algebra $\Z_q$ modulo products, lower depths, and quasi-modular forms. This leads to the notion of a bigraded space $\lqq$, which roughly consists of non-commutative polynomials in infinitely many letters $b_i$, $i\geq0$, such that
\begin{itemize}
\item[(i)] $\psi$ has zero coefficient at $b_0$ and no depth $1$ terms of even weight,
\item[(ii)] $\psi$ is primitive for the shuffle coproduct,
\item[(iii)] $\psi$ is invariant under some involution $\tau$.
\end{itemize}
The main results in this work are the following.
\begin{thm} \label{main_thm} 
\begin{itemize}  \setlength\itemsep{0.8em}
\item[(1)] The pair $(\lqq,\ariemp)$ is a Lie algebra.
\item[(2)] There is a non-trivial derivation $\delta$ on the Lie algebra $(\lqq,\ariemp)$, which increases the weight by $2$.
\item[(3)] We have an embedding of Lie algebras
$\theta: (\ls,\{-,-\})\hookrightarrow (\lqq,\ariemp)$.
\end{itemize}
\end{thm}
The Lie bracket $\ariemp$ is introduced in Definition \ref{def:Lie_bracket_lq}, and the derivation $\delta$ in Definition \ref{def:deriv}. By construction, we get a surjective algebra morphism
\begin{align} \label{eq:iso_grZq_Ulq}
\mathcal{U}(\lqq)^\vee \twoheadrightarrow\grDZ_q/\quasimod.
\end{align}
In contrast to \eqref{eq:iso_grZ_Uls}, this map is only a surjection and not an isomorphism. In other words, computing the dimensions of the Lie algebra $\lqq$ provides upper bounds for the weight- and depth-homogeneous components of $\Z_q$, but these bounds are not sharp in general.

\vspace{0,3cm}
\textbf{The Lie algebra $\lqq$ and bimoulds.} The theory of bimoulds was introduced in \cite{Ec11} and further developed in \cite{Sc15}. Roughly, a bimould is a collection $M=(M_d)_{d\geq0}$ of polynomials $M_d\in\QQ[X_1,Y_1,\ldots,X_d,Y_d]$ in $2d$ variables. By \cite{Sc15}, the space $\BARI$ of alternal and swap invariant bimoulds (with even depth $1$ component)\footnote{An explanation of these properties is given in Section \ref{sec:bimoulds}.} equipped with the ari bracket is a Lie algebra. In \cite{Sc15}, the space $\BARI$ just occurs as a natural extension of a Lie algebra of moulds related to multiple zeta values. But as also mentioned in \cite{Kü19}, it also encodes depth-graded algebraic relations in the algebra $\Z_q$. We show in Section \ref{sec:bimoulds} that there is an isomorphism of Lie algebras
\[(\lqq,\ariemp)\xrightarrow{\sim}(\BARI,\operatorname{ari}).\]

\vspace{0,3cm}
\textbf{Acknowledgment.} I would like to thank Ulf Kühn and Henrik Bachmann for comments on earlier versions of this work and related discussions. This project was supported by JSPS KAKENHI Grant 24KF0150.

\section{Balanced multiple $q$-zeta values} \label{sec:balanced_qmzv}

In this section we review balanced multiple $q$-zeta values as introduced in \cite{Bu24}, to give a precise and clear description of the relations in the algebra $\Z_q$. We will not give the explicit construction, but instead explain the algebraic structure in detail.

Let $\B=\{b_0,b_1,b_2,\ldots\}$ be an alphabet, whose elements we call letters. By $\QB$ we denote the free non-commutative algebra generated by $\B$. To the monic monomials in $\QB$ we refer as words, and the empty word we denote by $1$. We write $\B^*$ fo the set of all words with letters in $\B$ (including the empty word $1$). For any $\QQ$-algebra $R$, we let $\RB$ be the algebra of non-commutative power series over $\B$ with coefficients in $R$.
\begin{defi} For a word $w=b_{s_1}\cdots b_{s_l}\in \QB$, we define the \emph{weight} and \emph{depth} as
\begin{align} \label{eq:wt_dep_QB}
\wt(w)=s_1+\cdots+s_l+\#\{i\mid s_i=0\}, \qquad \dep(w)=l-\#\{i\mid s_i=0\}.
\end{align}
We regard the empty word $1$ to be of weight and depth $0$.
\end{defi}
Both notions define a grading on the vector space $\QB$, we denote the homogeneous subspace of weight $k$ and depth $d$ by $\QB_{k,d}$. So,
\begin{align*}
\QB=\bigoplus_{k,d\geq0} \QB_{k,d}.
\end{align*}
In the following, we will use this notation also for subspaces of $\QB$.

\begin{defi} \label{def:bal_quasish}
The \emph{balanced quasi-shuffle product} on $\QB$ is recursively defined by $1\ast_b w=w\ast_b 1=w$ and 
\begin{align} \label{eq:bal_qsh}
	b_iv\ast_b b_jw=b_i(v\ast_b b_jw)+b_j(b_iv\ast_b w)+\delta_{ij>0} b_{i+j}(v\ast_b w)
\end{align}
for $v,w\in \QB$. Here, $\delta_{ij>0}$ is equal to $1$ if $ij>0$ and $0$ else. 
\end{defi}
Restricting to the letters $b_i$ for $i>0$, we obtain the usual stuffle product for multiple zeta values. Just allowing the letters $b_0,b_1$, we recover the usual shuffle product. Therefore, we call this product 'balanced'. Observe that the balanced quasi-shuffle product $\ast_b$ is graded for the weight and filtered for the depth.

Denote by $\QB^0$ the subspace spanned by words, which do not end with $b_0$. This subspace is preserved by the balanced quasi-shuffle product $\ast_b$, so $\QB^0$ is a subalgebra of $(\QB,\ast_b)$.

\begin{defi} \label{def:tau}
	We define the involution $\tau:\QB^0\to \QB^0$ by 
\begin{align*}
	\tau(b_0^{m_1}b_{k_1}\cdots b_0^{m_d}b_{k_d})=b_0^{k_d-1}b_{m_d+1}\cdots b_0^{k_1-1}b_{m_1+1},
\end{align*} 
where $k_1,\ldots,k_d\geq1,\ m_1,\ldots,m_d\geq0$.
\end{defi}

Note that the map $\tau$ is homogeneous for both the weight and the depth.

\begin{thm} \label{thm:bal_qmzv} (\cite{Bu24}) There is a $\tau$-invariant and surjective algebra morphism
	\[(\QB^0,\ast_b)\to \Z_q,\quad b_{s_1}\cdots b_{s_l}\mapsto \zeta_q(s_1,\ldots,s_l).\]
\end{thm}
We call the images $\zeta_q(s_1,\ldots,s_l)$ the \emph{balanced multiple $q$-zeta values}. The map in Theorem \ref{thm:bal_qmzv} induces a notion of weight and depth on balanced multiple $q$-zeta values, we say that $\zeta_q(s_1,\ldots,s_l)$ is of weight $s_1+\cdots+s_l+\#\{i\mid s_i=0\}$ and depth $l-\#\{i\mid s_i=0\}$. 

For example, we have for $k\geq2$
\begin{align*}
\zeta_q(k)=-\frac{B_k}{2k!}+\frac{1}{(k-1)!}\sum_{m,n>0} n^{k-1}q^{mn},
\end{align*}
which equals the Fourier expansion of the classical Eisenstein series of weight $k$. Since $\Z_q$ is an algebra, we deduce that the algebra $\quasimod$ of all quasi-modular forms is contained in $\Z_q$. Note that the weight of quasi-modular forms and the weight of balanced multiple $q$-zeta values are compatible.

Similar to \cite[Remark 6.11.]{BB23} and as an analogue of \cite[Conjecture 1]{IKZ06}, we expect that the kernel of the map in Theorem \ref{thm:bal_qmzv} is fully generated by the $\tau$-invariance.

\begin{con} \label{con:Zq_all_rel} All algebraic relations in the algebra $\Z_q$ are a consequence of the balanced quasi-shuffle product and the $\tau$-invariance of the balanced multiple $q$-zeta values.
\end{con} 

Since the balanced quasi-shuffle product and the map $\tau$ are weight-homogeneous, Conjecture \ref{con:Zq_all_rel} would imply that the algebra $\Z_q$ is graded by weight
\begin{align*}
\Z_q \overset{?}{=} \bigoplus_{k\geq0} \Z_{q,k},
\end{align*}
where $\Z_{q,k}$ is spanned by all balanced multiple $q$-zeta values of weight $k$. The balanced quasi-shuffle product is only filtered for the depth, this implies that $\Z_q$ cannot be graded by depth. In this paper, we focus on the associated depth-graded relations. This means, we define a increasing filtration on the algebra $\Z_q$ given by $\fil_d\Z_q=0$ for $d<0$ and
\begin{align*}
\fil_d \Z_q=\operatorname{span}_\QQ\{\zeta_q(s_1,\ldots,s_l)\mid l-\#\{i\mid s_i=0\}\leq d\}, \qquad d\geq0.
\end{align*}
Then, we want to study the relations in the algebra
\begin{align} \label{eq:def_grDZq}
\grDZ_q=\bigoplus_{d\geq0} \big(\fil_d \Z_q/ \fil_{d-1}\Z_q\big).
\end{align}

\section{The definition of the space $\lqq$} \label{sec:def_lq}

In this section we will explain and motivate the definition of the space $\lqq$. A similar consideration for multiple zeta values is described in \cite{ENR03}. Then in Section \ref{sec:lq_Lie_alg} and \ref{sec:lq_Lie_alg_proof}, we prove that this space is equipped with a Lie algebra structure.

\vspace{0,2cm}
By \cite[Theorem 4.3]{Bu24-2}, there is a unique extension of the morphism in Theorem \ref{thm:bal_qmzv}
\begin{align}
\zeta_q^{\operatorname{reg}}: \QB\to \Z_q,\quad w\mapsto \zeta_q^{\operatorname{reg}}(w),
\end{align}
such that $\zeta_q^{\operatorname{reg}}(w)=\zeta_q(w)$ for $w\in \QB^0$ and $\zeta_q^{\operatorname{reg}}(b_0)=0$.
Denote by $\overline{\zeta_q}(w)$ the equivalence class of $\zeta_q^{\operatorname{reg}}(w)$ in the quotient 
\begin{align*} \label{eq:quotient_space}
R_{\overline{\zeta_q}}=\grD\Big(\Z_q/\big(\quasimod\cup \Z_{q,>0}^2\big)\Big).
\end{align*}
By Conjecture \ref{con:Zq_all_rel}, all relations in $\Z_q$ are obtained from the balanced quasi-shuffle product and the $\tau$-invariance of balanced multiple $q$-zeta values. To obtain the definition of $\lqq$, we reformulate this conjecture in terms of the generating series
\begin{align}
\Phi_{\overline{\zeta_q}}=\sum_{w\in \B^*} \overline{\zeta_q}(w)\ \!w \quad \in R_{\overline{\zeta_q}}\langle\langle \B\rangle\rangle.
\end{align}
The first step is to determine the associated depth-graded of the balanced quasi-shuffle product and the involution $\tau$.
As for the space $\Z_q$, we define a increasing filtration on the space $\QB$ by $\fil_d\QB=0$ for $d<0$ and 
\begin{align*}
\fil_d \QB=\operatorname{span}_\QQ\{w\mid \dep(w)\leq d\}, \qquad d\geq0.
\end{align*}
If we choose $\phi\in \fil_d\QB$ in the following, we always assume $d$ to be minimal, i.e., we assume that $\phi$ contains at least one word of depth $d$.

For $\phi\in \QB$ and $d\geq0$, we denote by $\phi_d$ the projection of $\phi$ to the homogeneous component $\QB_d$ of depth $d$. So, any $\phi\in \fil_d\QB$ can be uniquely written as
\[\phi=\sum_{l=0}^{d} \phi_l,\quad \phi_d\neq 0.\]
Then, there is a surjective map
\begin{align} \label{eq:def_grD}
\grD:\QB&\to\bigoplus_{d\geq0} \big(\fil_d \QB /\fil_{d-1} \QB\big) \simeq \QB, \\
\sum_{l=0}^d \phi_l&\mapsto \phi_d. \nonumber
\end{align}

Recall that the shuffle product $\shuffle$ on $\QB$ is defined by $1\shuffle w=w\shuffle 1=w$ and 
\begin{align} \label{eq:shuffle}
b_iv\shuffle b_j w= b_i(v\shuffle b_j w)+b_j(b_iv\shuffle w),\qquad v,w\in\QB.
\end{align}

\begin{lem} \label{lem:depthgr_balqsh_tau}
We have the following equalities of maps on $\QB^{\otimes 2}$ resp. $\QB^0$
\begin{flalign*}
\qquad \text{(i)} \quad & \grD \circ \ast_b=\shuffle, && \\
\qquad \text{(ii)} \quad & \grD \circ  \tau=\tau. &&
\end{flalign*}
\end{lem}
\begin{proof} For (i), we observe that the third term in \eqref{eq:bal_qsh} is always of lower depth an hence vanishes under $\grD$. The remaining is the recursive formula for the shuffle product $\shuffle$. (ii) follows from the observation that $\tau$ is depth homogeneous. 
\end{proof}

Lemma \ref{lem:depthgr_balqsh_tau} implies that \begin{align*}
\overline{\zeta_q}(v\shuffle w)&=0, \qquad v,w\in\QB\backslash\mathbb{Q},\\
\overline{\zeta_q}(\tau(w))&=\overline{\zeta_q}(w),\qquad w\in\QB^0.
\end{align*} 
The graded dual to the shuffle product is the coproduct $\shco:\QB\to \QB\otimes \QB$, which is on the letters given by
\begin{align} \label{eq:shco_B}
\shco(b_i)=b_i\otimes 1 +1 \otimes b_i, \quad i\geq0,
\end{align}
and extended with the concatenation product. Let $R$ be any $\QQ$-algebra and extend $\shco$ by to the algebra $\RB$ of power series. By construction, a $\QQ$-linear map $\varphi:\QB\to R$ vanishes on shuffle products, i.e., 
\begin{align} \label{eq:primitive_zero_on_prod}
\varphi(v\shuffle w)=0 \quad \text{ for all } v,w\in\QB\backslash \QQ,
\end{align} 
if and only if $\Phi=\sum\limits_{w\in \B^*} \varphi(w)w \in \RB$ is primitive for $\shco$, i.e., 
\[\shco \Phi=\Phi\otimes 1 +1 \otimes \Phi.\]
We apply this to the elements $\overline{\zeta_q}(w)$.
\begin{prop} \label{prop:Phi_zeta_primitive}
The series $\Phi_{\overline{\zeta_q}}$ is primitive for $\shco$,
\[\shco \Phi_{\overline{\zeta_q}}=\Phi_{\overline{\zeta_q}}\otimes 1 +1\otimes \Phi_{\overline{\zeta_q}}.\]
\end{prop}
By Lemma \ref{lem:depthgr_balqsh_tau}, the $\tau$-invariance does not change in the depth-graded case or modulo products. Let
\[\pizero:\QB\to\QB^0\]
be the canonical projection to $\QB^0$. For any $\QQ$-algebra $R$, extend $\pizero$ the algebra $\RB$ of power series, and similarly $\tau$ to the algebra $\RB^0$ of power series, which only have nonzero coefficients for words in $\QB^0$.
\begin{prop} \label{prop:Phi_zeta_tauinv}
We have
\[\tau(\pizero(\Phi_{\overline{\zeta_q}}))=\pizero(\Phi_{\overline{\zeta_q}}). \]
\end{prop}

For any power series $\Phi\in\RB$ and any word $w\in\QB$, we denote by $(\Phi\mid w)$ the coefficient of $w$ in $\Phi$ and extend this $R$-linearly. Observe that $\zeta_q^{\operatorname{reg}}(b_0)=0$ implies $\overline{\zeta_q}(b_0)=0$, and hence
\begin{align} \label{eq:Phi_zeta_b0}
(\Phi_{\overline{\zeta_q}}\mid b_0) =0.
\end{align}
Moreover, we consider modulo quasi-modular forms, so $\overline{\zeta_q}(b_0^mb_k)=0$ for $k+m$ even, or equivalently,
\begin{align} \label{eq:Phi_zero_quasimod}
(\Phi_{\overline{\zeta_q}}\mid b_0^mb_k)=0, \quad k+m \text{ even, } k\geq 1,\ m\geq0.
\end{align}
Proposition \ref{prop:Phi_zeta_primitive}, \ref{prop:Phi_zeta_tauinv} together with \eqref{eq:Phi_zeta_b0}, \eqref{eq:Phi_zero_quasimod} motivate the definition of $\lqq$.
\begin{defi} \label{def:lq} Let $\lqq$ be the space of all non-commutative polynomials $\Phi\in \QB$, such that
\begin{flalign*}
\qquad \text{(i)} \quad & (\Phi\mid b_0)=0, &&\\
\qquad \text{(ii)} \quad & (\Phi\mid b_0^mb_k)=0, \qquad k+m \text{ even},\ k\geq1,\ m\geq0, &&\\
\qquad \text{(iii)} \quad & \shco\Phi=\Phi\otimes 1+1\otimes \Phi, &&\\
\qquad \text{(iv)} \quad & \tau\big(\pizero(\Phi)\big)=\pizero(\Phi).
\end{flalign*}
\end{defi}

Note that for simplicity we consider polynomials with coefficients in $\QQ$ instead of power series with coefficients in some $\QQ$-algebra $R$ in Definition \ref{def:lq}. By passing to the completion $\lqq\widehat{\otimes} R$ of $\lqq\otimes R$, we obtain the space of power series in $\RB$ satisfying the conditions (i)-(iv). In particular, we have
\begin{align} \label{eq:Phi_zeta_in_lq}
\Phi_{\overline{\zeta_q}}\in \lqq\ \! \widehat{\otimes}\ \!R_{\overline{\zeta_q}}.
\end{align}
By \cite[Theorem 1.4]{Re93}, condition (iii) in Definition \ref{def:lq} is equivalent to being a Lie polynomial. Thus,
\begin{align} \label{eq:lq_in_Lie}
\lqq\subset \LieB,
\end{align}
where $\LieB$ denotes the free Lie algebra on the set $\B$.
Moreover, the space $\lqq$ is bigraded with respect to the weight and depth
\[\lqq=\bigoplus_{k,d\geq0} \lqq_{k,d}.\]
A reformulation of \eqref{eq:Phi_zeta_in_lq} is the following.

\begin{prop} \label{prop:surj_lqdual_indec} We have a surjective $\QQ$-linear map
\begin{align*} 
\lqq^\vee \twoheadrightarrow \grD\Z_q/\big(\quasimod\cup \Z_{q,>0}^2\big).
\end{align*}
\end{prop}

The space $\lqq^\vee$ denotes the graded dual with respect to the weight. So $\lqq^\vee=\bigoplus_{k\geq0} \lqq_k^*$, where $\lqq_k^*$ is the usual dual of a vector space. 

\begin{rem} \label{rem:lq_grDZ_q_not_inj}
In \cite{Bu23}, there is an explicit description of a space $\bmzero$, which we expect to be a Lie algebra and to be equipped with an isomorphism
\begin{align*}
\bmzero^\vee\overset{?}{\simeq} \Z_q/\big(\cup \Z_{q,>0}^2\big).
\end{align*} 
Moreover, we have a natural embedding
\begin{align*} 
\grD\bmzero\hookrightarrow\lqq.
\end{align*}
This map is not an isomorphism. For example, one checks numerically that the element
\begin{align*}
&b_3b_0b_2b_0b_0-b_2b_0b_3b_0b_0+b_2b_0b_0b_3b_0-b_3b_0b_0b_2b_0-b_0b_0b_2b_3b_0+b_0b_0b_3b_2b_0+b_0b_2b_3b_0b_0 \\	&-b_0b_3b_2b_0b_0-b_0b_0b_3b_0b_2+b_0b_0b_2b_0b_3-b_0b_2b_0b_0b_3+b_0b_3b_0b_0b_2 \in\lqq
\end{align*}
is not the depth-graded part of an element in $\bmzero$. This implies, that also the map in Proposition \ref{prop:surj_lqdual_indec} cannot be an isomorphism. A more detailed explanation of the structure of $\grD\bmzero$ inside $\lqq$ in the language of bimoulds (cf Section \ref{sec:bimoulds}) is given in \cite{Kü19}.
\end{rem} 

\begin{thm} \label{thm:lq_parity} If $k\not\equiv d \mod 2$, then we have
\[\lqq_{k,d}=\{0\}.\]
\end{thm}

\begin{proof} Let $\Phi\in \lqq_{k,d}$ be nonzero. In Section \ref{sec:bimoulds} we will introduce a bimould $\bimap(\Phi)$ corresponding to $\Phi$. By construction, the only non-trivial component of $\bimap(\Phi)$ is $\bimap(\Phi)_d$, which is a homogeneous polynomial of degree $k-d$. By Proposition \ref{prop:bimap_vsiso} and Proposition \ref{prop:bari_parity}, the degree $k-d$ must be even. This means, we have $k\equiv d \mod 2$.
\end{proof}

\begin{rem} The proof of Theorem \ref{thm:lq_parity} requires the theory of bimoulds so far. It would be interesting to find a direct proof inside $\QB$.
\end{rem}

\section{The Lie algebra structure on $\lqq$} \label{sec:lq_Lie_alg}

In this section, we give a proof that the space $\lqq$ is equipped with a bigraded Lie algebra structure. First, we introduce the Lie bracket, which comes from Ecalle's theory of bimoulds as discussed in Section \ref{sec:bimoulds}. The proof that this bracket endows $\lqq$ with a Lie algebra structure is motivated from the proof of \cite[Theorem 2.5.6]{Sc15}.

\begin{defi} \label{def:Lie_bracket_lq}
For any word $w=b_0^{m_1}b_{k_1}\dots b_0^{m_d}b_{k_d}b_0^{m_{d+1}}$ in $\QB$, we define the derivation $\der{w}$ on $\QB$ (with respect to concatenation) by 
\begin{align*}
\der{w}(b_0)&=0, \\
\der{w}(b_i)&=\sum_{\substack{l_1+\cdots+l_d+l=k_1+\cdots+k_d\\ l_1,\ldots,l_d\geq1,\ l\geq0}}(-1)^l\prod_{s=1}^d \binom{k_s-1}{l_s-1} \big[b_{i+l},b_0^{m_1}b_{l_1}\dots b_0^{m_d}b_{l_d}b_0^{m_{d+1}}\big], \qquad i\geq1.
\end{align*}
For $f,g\in\QB$, we set 
\begin{align} \label{eq:def_liebracket_lq}
\ari{f}{g}=\der{f}(g)-\der{g}(f)+[f,g].
\end{align}
\end{defi}

\begin{prop} \label{prop:ari_post_lie}
(\cite[Theorem 4.13]{BK25}) The tuple $(\LieB,[-,-],\der{})$ is a post-Lie algebra. Hence, $(\LieB,\ariemp)$ is a Lie algebra.
\end{prop} 

In particular, $\ariemp$ indeed satisfies anti-symmetry and the Jacobi identity. This means, also the larger space $(\QB,[-,-],\der{})$ forms a post-Lie algebra (cf.  \cite[Key Lemma 3.18]{Bu23}). As part (i) of Theorem \ref{main_thm}, we have the following.

\begin{thm} \label{thm:lqq_Lie_alg} The pair $(\lqq,\ariemp)$ is a bigraded Lie algebra.
\end{thm}

\begin{proof} 
It follows immediately from Definition \ref{def:Lie_bracket_lq} that the Lie bracket $\ariemp$ is homogeneous for the weight and depth. In particular, any Lie bracket $\ari{f}{g}$ with $f,g\in \QB$ of depth $\geq1$ has only terms of depth $\geq 2$. Since the conditions (i) and (ii) in the Definition \ref{def:lq} of $\lqq$ only involve terms of depth $1$, those are preserved by the Lie bracket $\ariemp$. 
	
Next, let $\Phi\in \QB$ satisfy condition (iii), so $\shco\Phi=\Phi\otimes 1+1\otimes\Phi$. This is equivalent to $\Phi\in \LieB$ by \eqref{eq:lq_in_Lie}. This condition is evidently preserved by the Lie bracket $\ariemp$, this follows directly from Definition \ref{def:lq} or Proposition \ref{prop:ari_post_lie}.

It remains to show that the Lie bracket $\ariemp$ preserves condition (iv) in Definition \ref{def:lq}. Proposition \ref{prop:lq_rho_inv} below shows that for any $\Phi\in \lqq$ the projection $\pizero(\Phi)$ is $\rho$-invariant. Hence, we can apply Corollary \ref{cor:Lie_brack_tau_inv} to $\Phi,\Psi\in \lqq$, and get
\[\tau\big(\arizero{\tau(\pizero(\Phi))}{\tau(\pizero(\psi))}\big)=\arizero{\pizero(\Phi)}{\pizero(\Psi)}.\]
We have $\tau(\pizero(\Phi))=\pizero(\Phi)$ and $\tau(\pizero(\Psi))=\pizero(\Psi)$, hence we deduce with \eqref{eq:Lie_brack_QB0} that
\[\tau\big(\pizero\big(\ari{\Phi}{\Psi}\big)\big)=\pizero\big(\ari{\Phi}{\Psi}\big).\qedhere\]
\end{proof}

\begin{rem} Conditions (ii) in the Definition \ref{def:lq} of $\lqq$ is necessary for $\ariemp$ preserving $\lqq$. So in other words, we have to consider $\Z_q$ modulo quasi-modular forms in \eqref{eq:quotient_space}. This could be an analogue of the fact that the Goncharov coproduct is expected to be only well-defined on the quotient of $\Z$  by $\zeta(2)$. 
\end{rem}

As $\lqq$ is a Lie algebra, we get the following refinement of Proposition \ref{prop:surj_lqdual_indec}.

\begin{thm} \label{thm:qmzv_depthgraded_surj} We have a surjective morphism of algebras
\[\mathcal{U}(\lqq)^\vee\twoheadrightarrow \grDZ_q/\quasimod.\]
\end{thm}

A coproduct on $\mathcal{U}(\lqq)^\vee$ is explicitly computed in \cite[Theorem 4.25]{BK25}, we expect this coproduct to descend to $\grDZ_q/\quasimod$. Remark \ref{rem:lq_grDZ_q_not_inj} implies that the morphism in Theorem \ref{thm:qmzv_depthgraded_surj} is not an isomorphism.

\section{Proof of Theorem \ref{thm:lqq_Lie_alg}} \label{sec:lq_Lie_alg_proof}

In this section, we prove the results applied in the proof of Theorem \ref{thm:lqq_Lie_alg}.
We will need an inverse of the projection $\pizero:\QB\to \QB^0$, in particular, as the involution $\tau$ is only defined on the subspace $\QB^0$.

\begin{defi} \label{def:sec} Let $\sec:\QB^0\to \QB$ be the $\QQ$-linear map given by
\[\sec(b_0^{m_1}b_{k_1}\cdots b_0^{m_d}b_{k_d})=\ad(b_0)^{m_1}\Big(b_{k_1}\cdot \ad(b_0)^{m_2}\Big(b_{k_2}\cdots \ad(b_0)^{m_d}(b_{k_d})\cdots\Big)\Big),\]
where $\ad(b_0):\QB\to \QB, w\mapsto[b_0,w]$ is meant with respect to the usual commutator.
\end{defi}

Explicitly, we have for $k_1,\ldots,k_d\geq1$, $m_1,\ldots,m_d\geq0$
\begin{align} \label{eq:sec_expl}
\sec(b_0^{m_1}b_{k_1}\cdots b_0^{m_d}b_{k_d})=\sum_{\substack{n_1+\cdots+n_d+n=m_1+\cdots+m_d \\ n_1,\ldots,n_d,n\geq0}} (-1)^n \prod_{s=1}^d \binom{m_s}{n_s} b_0^{n_1}b_{k_1}\cdots b_0^{n_d}b_{k_d}b_0^n.
\end{align}

Let $\partial_0:\QB\to \QB$ be the derivation with respect to concatenation given by 
\begin{align} \label{eq:def_d0}
\partial_0(b_0)=1,\qquad \partial_0(b_i)=0,\ i\geq1.
\end{align}
We abbreviate 
\begin{align} \label{eq:def_Dbi}
D_{k,m}=\ad(b_0)^m(b_k),\quad k\geq1,\ m\geq0,
\end{align}
and write $\Dbi=\{D_{k,m}\mid k\geq1,m\geq0\}$. By \eqref{eq:lq_in_Lie} and condition (i) in Definition \ref{def:lq}, we have 
\begin{align} \label{eq:lq_in_QDbi}
\lqq\subset \QDbi.
\end{align} 

\begin{prop} \label{prop:pi0_sec_inverse} We have
\begin{enumerate}[(i)] \setlength\itemsep{0.4em}
\item $\sec(w) = \sum\limits_{n\geq0} \frac{(-1)^n}{n!} \partial_0^n(w)b_0^n,\quad$ $w\in \QB^0$,
\item $\pizero\circ \sec=\id$,
\item $\sec\circ \pizero=\id$ on the subspace $\QDbi\subset \QB$.
\end{enumerate}
\end{prop}

Note that by Proposition \ref{prop:pi0_sec_inverse} and \eqref{eq:lq_in_QDbi}, we can uniquely recover $\Phi\in \lqq$ from $\pizero(\Phi)$.

\begin{proof}
The equality in (i) is obtained by direct calculation. For $w=b_0^{m_1}b_{k_1}\cdots b_0^{m_d}b_{k_d}$, $k_1,\ldots,k_d\geq1,\ m_1,\ldots,m_d\geq0$, we have
\begin{align*}
\sum_{n\geq0} \frac{(-1)^n}{n!} \partial_0^n(w)b_0^n&=\sum_{n\geq0}\sum_{\substack{n_1+\cdots+n_d=n \\ 0\leq n_s\leq m_s}} \frac{(-1)^n}{n_1!\cdots n_d!} \prod_{s=1}^d \frac{m_s!}{(m_s-n_s)!} b_0^{m_1-n_1}b_{k_1}\cdots b_0^{m_d-n_d}b_{k_d}b_0^n \nonumber\\
&=\sum_{n_1+\cdots+n_d+n=m_1+\cdots+m_d} (-1)^n \prod_{s=1}^d \binom{m_s}{n_s} b_0^{n_1}b_{k_1}\cdots b_0^{n_d}b_{k_d}b_0^n \nonumber\\
&=\sec(w),
\end{align*}
where the last equality follows from \eqref{eq:sec_expl}.
For (ii), observe that in \eqref{eq:sec_expl} the only term in $\QB^0$ is the one for $n=0$ and hence $n_s=m_s$ for $s=1,\ldots,d$. Thus, we deduce
\[\pizero(\sec(b_0^{m_1}b_{k_1}\cdots b_0^{m_d}b_{k_d}))=b_0^{m_1}b_{k_1}\cdots b_0^{m_d}b_{k_d}.\]
For (iii), we compute for $k\geq1,\ m\geq0$
\begin{align*}
\partial_0\big(\ad(b_0)^m(b_k)\big)&=\sum_{n=0}^m (-1)^{m+n} \binom{m}{n} \partial_0(b_0^nb_kb_0^{m-n}) \\
&=\sum_{n=1}^m (-1)^{m+n} \binom{m}{n} n b_0^{n-1}b_kb_0^{m-n} + \sum_{n=0}^{m-1} (-1)^{m+n} \binom{m}{n} (m-n)b_0^nb_kb_0^{m-n-1} \\
&=\sum_{n=1}^m (-1)^{m+n} \binom{m}{n} n b_0^{n-1}b_kb_0^{m-n} + \sum_{n=1}^{m} (-1)^{m+n-1} \binom{m}{n} nb_0^{n-1}b_kb_0^{m-n} \\
&=0.
\end{align*}
This calculation implies $\QDbi \subset \ker(\partial_0)$. So, assume that $w\in \QB$ satisfies $\partial_0(w)=0$. Write $w=\sum_{n\geq0} w_nb_0^n$ with $w_n\in\QB^0$. Then, we have
\[0=\partial_0(w)=\sum_{n\geq0} \partial_0(w_n)b_0^n+\sum_{n\geq1} n w_nb_0^{n-1},\qquad n\geq1.\]
Since $\partial_0(w_n)\in\QB^0$, coefficient comparison with respect to $b_0$ at the end leads to
\[\partial_0(w_{n-1})=-nw_n, \qquad n\geq1. \]
Inductively, we get
\begin{align} \label{eq:wn_and_partial0} 
w_n=\frac{(-1)^n}{n!} \partial_0^n(w_0),\qquad n\geq1.
\end{align}
 We deduce for $w\in \QDbi$  by combining (i) and \eqref{eq:wn_and_partial0} that
\[\sec(\pizero(w))=\sec(w_0)= \sum_{n\geq0} \frac{(-1)^n}{n!} \partial_0^n(w_0)b_0^n=\sum_{n\geq0} w_nb_0^n=w. \qedhere \]
\end{proof}

\begin{defi} \label{def:rho} Let $\rho:\QB^0\to \QB^0$ be the $\QQ$-linear map given by
\[\rho(b_0^{m_1}b_{k_1}\cdots b_0^{m_d}b_{k_d})=\hspace{-0.3cm}\sum_{\substack{l_1+\cdots+l_d=k_1+\cdots+k_d \\ n_1+\cdots+n_d=m_1+\cdots+m_d \\ l_s\geq1,\ n_s\geq0}} \hspace{-0.1cm} (-1)^{l_d+n_d-1}\prod_{s=1}^{d-1} \binom{k_s-1}{l_s-1}\binom{m_s}{n_s} b_0^{n_d}b_{l_d}b_0^{n_1}b_{l_1}\cdots b_0^{n_{d-1}}b_{l_{d-1}}.\]
\end{defi}

\begin{prop} \label{prop:lq_rho_inv} If $\Phi\in \lqq$, then $\rho(\pizero(\Phi))=\pizero(\Phi)$.
\end{prop}

\begin{proof} Let $S:\QB\to \QB$ be the antipode of the Hopf algebra $(\QB,\cdot,\shco)$, i.e., we have
\begin{align} \label{eq:antipode}
S(b_{s_1}\cdots b_{s_l})=(-1)^l b_{s_l}\cdots b_{s_1}.
\end{align}
By condition (iii) in Definition \ref{def:lq}, $\Phi\in \lqq$ is a primitive element of $(\QB,\cdot,\shco)$. Therefore, we have $S(\Phi)=-\Phi$. Thus by Proposition \ref{prop:pi0_sec_inverse}, the projection $\pizero(\Phi)\in \QB^0$ satisfies $S_0(\pizero(\Phi))=-\pizero(\Phi)$ where $S_0=\pizero\circ S\circ \sec$ is explicitly given by
\begin{align} \label{eq:S_0}
S_0(b_0^{m_1}b_{k_1}\cdots b_0^{m_d}b_{k_d})=(-1)^d \sum_{\substack{n_1+\cdots+n_d=m_1+\cdots+m_d \\ n_1,\ldots,n_d\geq 0}} \prod_{s=2}^d (-1)^{n_s} \binom{m_s}{n_s} b_0^{n_1}b_{k_d}b_0^{n_d}b_{k_{d-1}}\cdots b_0^{n_2}b_{k_1}.
\end{align}
We show that 
\begin{align*} 
\rho(\pizero(\Phi))=(S_0\circ \tau\circ S_0\circ \tau)(\pizero(\Phi)).
\end{align*}
Since we have for $\Phi\in \lqq$ that $\tau(\pizero(\Phi))=\pizero(\Phi)$ and $S_0(\pizero(\Phi))=-\pizero(\Phi)$, we then get $\rho(\pizero(\Phi))=\pizero(\Phi)$. Let $w=b_0^{m_1}b_{k_1}\cdots b_0^{m_d}b_{k_d}$ where $k_1,\ldots,k_d\geq1,\ m_1,\ldots,m_d\geq0$, then
\begin{align}  \label{eq:rho_equal_S_tau}
&S_0\circ \tau\circ S_0\circ \tau(w)=S_0\circ \tau\circ S_0(b_0^{k_d-1}b_{m_d+1}\cdots b_0^{k_1-1}b_{m_1+1}) \nonumber\\
&=(-1)^d \sum_{\substack{l_1+\cdots+l_d=k_1+\cdots+k_d \\ l_s\geq1}} \prod_{s=1}^{d-1} (-1)^{l_s-1} \binom{k_s-1}{l_s-1} S_0\circ\tau(b_0^{l_d-1}b_{m_1+1}b_0^{l_1-1}b_{m_2+1}\cdots b_0^{l_{d-1}-1}b_{m_d+1}) \nonumber\\
&=(-1)^d \sum_{\substack{l_1+\cdots+l_d=k_1+\cdots+k_d \\ l_s\geq1}} \prod_{s=1}^{d-1} (-1)^{l_s-1} \binom{k_s-1}{l_s-1} S_0(b_0^{m_d}b_{l_{d-1}}\cdots b_0^{m_2}b_{l_1}b_0^{m_1}b_{l_d}) \nonumber\\
&=\sum_{\substack{l_1+\cdots+l_d=k_1+\cdots+k_d \\ n_1+\cdots+n_d=m_1+\cdots+m_d \\ l_s\geq1,\ n_s\geq0}} \prod_{s=1}^{d-1} (-1)^{l_s+n_s-1}\binom{k_s-1}{l_s-1} \binom{m_s}{n_s} b_0^{n_d}b_{l_d}b_0^{n_1}b_{l_1}\cdots b_0^{n_{d-1}}b_{l_{d-1}} \nonumber\\
&=(-1)^{\wt(w)+\dep(w)}\rho(b_0^{m_1}b_{k_1}\cdots b_0^{m_d}b_{k_d}).
\end{align}
By Theorem \ref{thm:lq_parity}, we have $\lqq_{k,d}=\{0\}$ if $k\not\equiv d\mod 2$. In other words, elements in $\lqq$ consist of words $w\in \QB$ with $\wt(w)+\dep(w)\equiv 0\mod 2$. Thus, deduce from the previous calculation that we have $S_0\circ \tau\circ S_0\circ \tau(\pizero(\Phi))=\rho(\pizero(\Phi))$ for $\Phi\in\lqq$.
\end{proof}

For $w=b_0^{m_1}b_{k_1}\cdots b_0^{m_d}b_{k_d}b_0^{m_{d+1}}$, let $\derr{w},\ \derl{w}:\QB\to\QB$ be the derivations given by 
\begin{align}
\derr{w}(b_0)&=\derl{w}(b_0)=0, \nonumber \\
\label{eq:derr}
\derr{w}(b_i)&=\sum_{\substack{l_1+\cdots+l_d+l=k_1+\cdots+k_d\\ l_1,\ldots,l_d\geq1,\ l\geq0}}(-1)^l\prod_{s=1}^d \binom{k_s-1}{l_s-1} b_{i+l}b_0^{m_1}b_{l_1}\dots b_0^{m_d}b_{l_d}b_0^{m_{d+1}}, \qquad i\geq1, \\
\label{eq:derl}
\derl{w}(b_i)&=\sum_{\substack{l_1+\cdots+l_d+l=k_1+\cdots+k_d\\ l_1,\ldots,l_d\geq1,\ l\geq0}}(-1)^l\prod_{s=1}^d \binom{k_s-1}{l_s-1} b_0^{m_1}b_{l_1}\dots b_0^{m_d}b_{l_d}b_0^{m_{d+1}}b_{i+l}, \qquad i\geq1.
\end{align}
Then, we have
\begin{align} \label{eq:dA_in_dAl_dAr}
\der{w}=\derr{w}-\derl{w}.
\end{align}

For $w\in \QB^0$ and $\bullet\in \{l,r\}$, define $\derbullzero{w}:\QB^0\to \QB^0$ by $
\derbullzero{w}=\pizero\circ \derbull{\sec(w)}\circ\sec$. Then, by Proposition \ref{prop:pi0_sec_inverse}, we get
\begin{align*}
\derbullzero{\pizero(w)}(\pizero(v))=\pizero(\derbull{w}(v)), \qquad v,w\in \QDbi.
\end{align*}
Explicitly, we have for $k_1,\ldots,k_d,s_1,\ldots,s_e\geq1$, $m_1,\ldots,m_d,t_1,\ldots,t_e\geq0$
\begin{align} 
\label{eq:dAr0_expl}
\derrzero{b_0^{m_1}b_{k_1}\cdots b_0^{m_d}b_{k_d}}\big(b_0^{t_1}b_{s_1}\cdots b_0^{t_e}b_{s_e}\big)&=\sum_{i=1}^e \sum_{\substack{l_1+\cdots+l_d+l=k_1+\cdots+k_d \\ n_1+\cdots+n_d+n=m_1+\cdots+m_d \\ l_s\geq1,\ n_s,n,l\geq0, \ n=0 \text{ if } i=e}} (-1)^{l+n} \prod_{s=1}^d \binom{k_s-1}{l_s-1}\binom{m_s}{n_s} \\
&\cdot b_0^{t_1}b_{s_1}\cdots b_0^{t_{i-1}}b_{s_{i-1}}b_0^{t_i}\big(b_{s_i+l}b_0^{n_1}b_{l_1}\cdots b_0^{n_d}b_{l_d}b_0^n\big)b_0^{t_{i+1}}b_{s_{i+1}}\cdots b_0^{t_e}b_{s_e}, \nonumber \\
\label{eq:dAl0_expl}
\derlzero{b_0^{m_1}b_{k_1}\cdots b_0^{m_d}b_{k_d}}\big(b_0^{t_1}b_{s_1}\cdots b_0^{t_e}b_{s_e}\big)&=\sum_{i=1}^e \sum_{\substack{l_1+\cdots+l_d+l=k_1+\cdots+k_d \\ n_1+\cdots+n_d+n=m_1+\cdots+m_d \\ l_s\geq1,\ n_s,n,l\geq0}} (-1)^{l+n} \prod_{s=1}^d \binom{k_s-1}{l_s-1}\binom{m_s}{n_s} \\
&\cdot b_0^{t_1}b_{s_1}\cdots b_0^{t_{i-1}}b_{s_{i-1}}b_0^{t_i}\big(b_0^{n_1}b_{l_1}\cdots b_0^{n_d}b_{l_d}b_0^nb_{s_i+l}\big)b_0^{t_{i+1}}b_{s_{i+1}}\cdots b_0^{t_e}b_{s_e}. \nonumber 
\end{align}
We set
\begin{align} \label{eq:derzero}
\derzero{w}=\derrzero{w}-\derlzero{w},\qquad w\in\QB^0.
\end{align}

Moreover, for $v,w\in \QB^0$ set
\begin{align} \label{eq:conzero}
v \conczero w= \pizero\big(\sec(v)\sec(w)\big).
\end{align}
So as before, we have by Proposition \ref{prop:pi0_sec_inverse} for $v,w\in \QDbi$ that $\pizero(v)\conczero\pizero(w)=\pizero(vw)$.

\begin{prop} \label{prop:dA0_and_tau} For $w,v\in \QB^0$, we have
\begin{enumerate}[(i)]
\item $\big(\tau\circ \derrzero{\tau(w)}\circ \tau\big)(v)= \derrzero{w}(v)+w \conczero v-\tau\big(\tau(w)\conczero\tau(v)\big)$,
\item $\tau\circ \derlzero{\tau(w)}\circ \tau=\derlzero{\rho(w)}$.
\end{enumerate}
\end{prop}

The proof is given in Appendix \ref{ap:proof_prop_dA0_and_tau}.

\vspace{0,4cm}
For $v,w\in \QB^0$ write 
\begin{align} \label{eq:arizero}
\arizero{v}{w}&=\pizero\big(\ari{\sec(v)}{\sec(w)}\big) \nonumber \\
&=\derzero{v}(w)-\derzero{w}(v)+v\conczero w - w\conczero v.
\end{align}
Then by Proposition \ref{prop:pi0_sec_inverse}, we have for $v,w\in \QDbi$ that 
\begin{align} \label{eq:Lie_brack_QB0} \arizero{\pizero(v)}{\pizero(w)}=\pizero\big(\ari{v}{w}\big).
\end{align}

\begin{cor} \label{cor:Lie_brack_tau_inv} If $v,w\in \QB^0$ are $\rho$-invariant, then we have
	\[\tau\big(\arizero{\tau(v)}{\tau(w)}\big)=\arizero{v}{w}.\]
\end{cor}

\begin{proof} For any $v,w\in \QB^0$, we have by Definition \ref{def:Lie_bracket_lq} and \eqref{eq:dA_in_dAl_dAr}
\begin{align*}
\arizero{v}{w}&=\derzero{v}(w)-\derzero{w}(v)+v\conczero w - w\conczero v \\ &=\derrzero{v}(w)-\derlzero{v}(w)-\derrzero{w}(v)+\derlzero{w}(v)+v\conczero w - w\conczero v.
\end{align*}
Applying Proposition \ref{prop:dA0_and_tau} yields
\begin{align*}
\tau\big(\arizero{\tau(v)}{\tau(w)}\big)&=\tau\Big(\derrzero{\tau(v)}(\tau(w))-\derlzero{\tau(v)}(\tau(w))-\derrzero{\tau(w)}(\tau(v))+\derlzero{\tau(w)}(\tau(v)) \\
&\hspace{7.5cm}+\tau(v)\conczero \tau(w) - \tau(w)\conczero \tau(v)\Big)\\
&=\derrzero{v}(w)+v\conczero w-\tau\big(\tau(v)\conczero\tau(w)\big)-\derlzero{\rho(v)}(w)-\derrzero{w}(v)-w\conczero v \\
&\hspace{1.5cm}+\tau\big(\tau(w)\conczero\tau(v)\big)+\derlzero{\rho(w)}(v)+\tau\big(\tau(v)\conczero \tau(w)\big) - \tau\big(\tau(w)\conczero \tau(v)\big) \\
&=\derrzero{v}(w)-\derlzero{\rho(v)}(w)-\derrzero{w}(v)+\derlzero{\rho(w)}(v)+v\conczero w-w\conczero v.
\end{align*}
If $v,w\in \QB^0$ are $\rho$-invariant, we deduce with \eqref{eq:dA_in_dAl_dAr}
\begin{align*}
\tau\big(\arizero{\tau(v)}{\tau(w)}\big)=\derzero{v}(w)-\derzero{w}(v)+v\conczero w-w\conczero v=\arizero{v}{w}.\hspace{4.2cm}
\qedhere\end{align*}
\end{proof}

\section{A derivation on $\lqq$} \label{sec:deriv_on_lq}

For this section, it is convenient to describe $\lqq$ in terms of the alphabet $\Dbi$ from \eqref{eq:def_Dbi}.

\vspace{0.3cm}
By Lazard elimination, we have $\LieB=\LieDbi\oplus \QQ b_0$.
Thus, by condition (i) in Definition \ref{def:lq} and \eqref{eq:lq_in_Lie} we deduce
\[\lqq\subset\LieDbi.\]
We extend the map $\tau$ to $\QDbi$ with the map $\sec$ from Definition \ref{def:sec},
\begin{align} \label{eq:def_tau_Dbi}
\tau_{\Dbi}= \sec \circ\ \!\tau \circ \pizero.
\end{align}
For $k_1,\ldots,k_d\geq1,\ m_1,\ldots,m_d\geq0$, we have with $n_d:=m_d$
\begin{align*}
&\tau_{\Dbi}(D_{k_1,m_1}\cdots D_{k_d,m_d})=\sum_{\substack{l_1^{(s)}+\cdots+l_s^{(s)}=k_s-1,\ s=1,\ldots,d \\ n_1+\cdots+n_{d-1}+n=m_1+\cdots+m_d}} \prod_{s=1}^d \binom{k_s-1}{l_1^{(s)},\ldots,l_s^{(s)}} \binom{m_s}{n_s}(-1)^{m_s+n_s} \\
&\cdot D_{m_d+m_{d-1}-n_{d-1}+1,l_d^{(d)}}D_{n_{d-1}+m_{d-2}-n_{d-2}+1,l_{d-1}^{(d-1)}+l_{d-1}^{(d)}}\cdots D_{n_2+m_1-n_1+1,l_2^{(2)}+\cdots+l_2^{(d)}}D_{n_1+1,l_1^{(1)}+\cdots l_1^{(d)}}.
\end{align*}
Then alternatively, we could define $\lqq$ as follows.

\begin{lem} \label{lem:lqq_in_Dbi} The space $\lqq$ consists of all $\Phi\in \LieDbi$ such that
\begin{flalign*}
\qquad \text{(ii')} \quad & (\Phi\mid D_{k,m})=0, \qquad k+m \text{ even},\ k\geq1,\ m\geq0, &&\\
\qquad \text{(iv')} \quad & \tau_{\Dbi}(\Phi)=\Phi.
\end{flalign*}
\end{lem} 
\begin{proof}
The conditions (i) and (iii) in the Definition \ref{def:lq} of $\lqq$ translate by \eqref{eq:lq_in_Lie} into $\Phi$ being an element of $\LieDbi$. A reformulation of (ii) resp. (iv) in Definition \ref{def:lq} is by construction given by (ii') resp. (iv').
\end{proof}

The derivation $\der{w}$ from Definition \ref{def:Lie_bracket_lq} translates for $w=D_{k_1,m_1}\cdots D_{k_d,m_d}$ in $\QDbi$ to
\begin{align} \label{eq:def_Liebracket_lq}
\der{w}(D_{i,n})&=\sum_{\substack{l_1+\cdots+l_d+l=k_1+\cdots+k_d \\ l_1,\ldots,l_d\geq1,\ l\geq0}} (-1)^l\prod_{s=1}^d\binom{k_s-1}{l_s-1} \ad(b_0)^n\big([D_{i+l,0},D_{l_1,m_1}\cdots D_{l_d,m_d}]\big) \\
&=\sum_{\substack{l_1+\cdots+l_d+l=k_1+\cdots+k_d \\ p_1+\cdots+p_d+p=n \\ l_s\geq1,\ p_s,p,l\geq0}} (-1)^l\prod_{s=1}^d\binom{k_s-1}{l_s-1} \binom{n}{p_1,\ldots,p_d,p} [D_{i+l,p},D_{l_1,m_1+p_1}\cdots D_{l_d,m_d+p_d}], \nonumber
\end{align}
and the Lie bracket from \eqref{eq:def_liebracket_lq} is still written as
\begin{align*}
\ari{f}{g}=\der{f}(g)-\der{g}(f)+[f,g],\qquad f,g\in\QDbi.
\end{align*}

\begin{defi} \label{def:deriv} Let $\delta$ be the derivation on $\QDbi$ given by
\[\delta(D_{k,m})=D_{k+1,m+1},\qquad k\geq1,\ m\geq0.\]
\end{defi}

\begin{prop} \label{prop:d_tau_commute} The maps $\tau_{\Dbi}$ and $\delta$ commute,
\[\tau_{\Dbi}\circ\delta=\delta\circ \tau_{\Dbi}.\]
\end{prop}

\begin{proof}
Let $k_1,\ldots,k_d\geq1$, $m_1,\ldots,m_d\geq0$. We have with $n_d:=m_d$
\begin{align} \label{eq:lq_deriv_-1}
&\tau_{\Dbi}\big(\delta(D_{k_1,m_1}\cdots D_{k_d,m_d})\big)=\sum_{t=1}^d \sum_{\substack{l_1^{(s)}+\cdots+l_s^{(s)}=k_s-1,\ s\neq t \\ l_1^{(t)}+\cdots+l_t^{(t)}=k_t \\ n_1+\cdots+n_{d-1}+n=m_1+\cdots+m_d+1}} \prod_{\substack{s=1 \\ s\neq t}}^d \binom{k_s-1}{l_1^{(s)},\ldots,l_s^{(s)}} \binom{m_s}{n_s}(-1)^{m_s+n_s} \nonumber \\ &\binom{k_t}{l_1^{(t)},\ldots,l_t^{(t)}}\binom{m_t+1}{n_t}(-1)^{m_t+n_t+1}D_{m_d+m_{d-1}-n_{d-1}+1,l_d^{(d)}}D_{n_{d-1}+m_{d-2}-n_{d-2}+1,l_{d-1}^{(d-1)}+l_{d-1}^{(d)}} \nonumber \\
&\hspace{2cm}\cdots D_{n_{t+1}+m_t-n_t+2,l_{t+1}^{(t+1)}+\cdots+l_{t+1}^{(d)}}\cdots D_{n_2+m_1-n_1+1,l_2^{(2)}+\cdots+l_2^{(d)}}D_{n_1+1,l_1^{(1)}+\cdots l_1^{(d)}}.
\end{align}
Applying the recursion
\[\binom{k}{l_1,\ldots,l_t}=\binom{k-1}{l_1-1,l_2,\ldots,l_t}+\cdots+\binom{k-1}{l_1,\ldots,l_{t-1},l_t-1},\quad \binom{m+1}{n}=\binom{m}{n}+\binom{m}{n-1}\]
to \eqref{eq:lq_deriv_-1} yields together with variable substitutions $l_i^{(t)}-1\to l_i^{(t)}$ and $n_t-1\to n_t$
\begin{align} \label{eq:lq_deriv-2}
&\sum_{\substack{l_1^{(s)}+\cdots+l_s^{(s)}=k_s-1,\ s=1,\ldots,d \\ n_1+\cdots+n_{d-1}+n=m_1+\cdots+m_d}}  \prod_{s=1}^d \binom{k_s-1}{l_1^{(s)},\ldots,l_s^{(s)}} \binom{m_s}{n_s}(-1)^{m_s+n_s}\\
&\Bigg(\sum_{t=1}^dD_{m_d+m_{d-1}-n_{d-1}+1,l_d^{(d)}}D_{n_{d-1}+m_{d-2}-n_{d-2}+1,l_{d-1}^{(d-1)}+l_{d-1}^{(d)}}\cdots D_{n_{t+1}+m_t-n_t+1,l_{t+1}^{(t+1)}+\cdots+l_{t+1}^{(d)}} \nonumber \\
&\hspace{4cm} D_{n_t+m_{t-1}-n_{t-1}+2,l_t^{(t)}+\cdots+l_t^{(d)}+1}\cdots D_{n_2+m_1-n_1+1,l_2^{(2)}+\cdots+l_2^{(d)}}D_{n_1+1,l_1^{(1)}+\cdots l_1^{(d)}} \nonumber \\
&-\sum_{t=1}^{d-1}\sum_{i=1}^t D_{m_d+m_{d-1}-n_{d-1}+1,l_d^{(d)}}D_{n_{d-1}+m_{d-2}-n_{d-2}+1,l_{d-1}^{(d-1)}+l_{d-1}^{(d)}}\cdots D_{n_{t+1}+m_t-n_t+2,l_{t+1}^{(t+1)}+\cdots+l_{t+1}^{(d)}} \nonumber \\
&\hspace{3cm} \cdots D_{n_i+m_{i-1}-n_{i-1}+1,l_i^{(i)}+\cdots+l_i^{(d)}+1}\cdots D_{n_2+m_1-n_1+1,l_2^{(2)}+\cdots+l_2^{(d)}}D_{n_1+1,l_1^{(1)}+\cdots l_1^{(d)}} \nonumber \\
&+\sum_{t=2}^d\sum_{i=1}^{t-1} D_{m_d+m_{d-1}-n_{d-1}+1,l_d^{(d)}}D_{n_{d-1}+m_{d-2}-n_{d-2}+1,l_{d-1}^{(d-1)}+l_{d-1}^{(d)}}\cdots D_{n_t+m_{t-1}-n_{t-1}+2,l_t^{(t)}+\cdots+l_t^{(d)}} \nonumber \\
&\hspace{3cm} \cdots D_{n_i+m_{i-1}-n_{i-1}+1,l_i^{(i)}+\cdots+l_i^{(d)}+1}\cdots D_{n_2+m_1-n_1+1,l_2^{(2)}+\cdots+l_2^{(d)}}D_{n_1+1,l_1^{(1)}+\cdots l_1^{(d)}}\Bigg). \nonumber
\end{align}
The latter two sums in \eqref{eq:lq_deriv-2} form a telescoping sum, thus only the first terms survives, which is equal to
$\delta\big(\tau_{\Dbi}(D_{k_1,m_1}\cdots D_{k_d,m_d})\big)$. Thus, we showed that
\[(\tau_{\Dbi}\circ \delta)(D_{k_1,m_1}\cdots D_{k_d,m_d}) =(\delta\circ\tau_{\Dbi})(D_{k_1,m_1}\cdots D_{k_d,m_d}).\]
Since both $\tau_{\Dbi}$ and $\delta$ are $\QQ$-linear maps, the claim follows.
\end{proof}

As part (ii) of Theorem \ref{main_thm}, we have the following.

\begin{thm} \label{thm:lq_deriv} The tuple $(\lqq,\ariemp,\delta)$ is a differential Lie algebra.
\end{thm}

\begin{proof} Evidently, the derivation $\delta$ preserves the space $\LieDbi$ and the condition (ii') from Lemma \ref{lem:lqq_in_Dbi}. By Proposition \ref{prop:d_tau_commute}, $\delta$ also preserves the $\tau_{\Dbi}$-invariance given in in (iv'). 
We verify the compatibility of the Lie bracket $\ariemp$ and the derivation $\delta$,
\begin{align} \label{eq:lq_deriv_liebrack}
\delta\big(\ari{f}{g}\big)=\ari{\delta(f)}{g}+\ari{f}{\delta(g)},\qquad f,g\in\QDbi.
\end{align} 
By $\QQ$-linearity, we can assume that $f=D_{k_1,m_1}\cdots D_{k_d,m_d},\ g=D_{i_1,n_1}\cdots D_{i_r,n_r}$. Then, we have 
\begin{align} \label{eq:lq_deriv1}
&\delta\big(\der{f}(g)\big)= \sum_{\substack{s,t=1 \\ s\neq t}}^r D_{i_1,n_1}\cdots \der{f}(D_{i_s,n_s})\cdots \delta(D_{i_t,n_t})\cdots D_{i_r,n_r} \\
\label{eq:lq_deriv2}
&+\sum_{s=1}^r\sum_{\substack{l_1+\cdots+l_d+l=k_1+\cdots+k_d \\ p_1+\cdots+p_d+p=n_s}} (-1)^l\binom{k_1-1}{l_1-1}\cdots \binom{k_d-1}{l_d-1} \binom{n_s}{p_1,\ldots,p_d,p}  \\
&\hspace{2.4cm} D_{i_1,n_1}\cdots D_{i_{s-1},n_{s-1}}[D_{i_s+1+l,p+1},D_{l_1,m_1+p_1}\cdots D_{l_d,m_d+p_d}]D_{i_{s+1},n_{s+1}}\cdots D_{i_r,n_r}\big)  \nonumber \\
\label{eq:lq_deriv3}
&+\sum_{s=1}^r \sum_{t=1}^d\sum_{\substack{l_1+\cdots+l_d+l=k_1+\cdots+k_d \\ p_1+\cdots+p_d+p=n_s}} (-1)^l\binom{k_1-1}{l_1-1}\cdots\binom{k_d-1}{l_d-1} \binom{n_s}{p_1,\ldots,p_d,p} \\
&D_{i_1,n_1}\cdots D_{i_{s-1},n_{s-1}}[D_{i_s+l,p},D_{l_1,m_1+p_1}\cdots D_{l_t+1,m_t+1+p_1} \cdots D_{l_d,m_d+p_d}]D_{i_{s+1},n_{s+1}}\cdots D_{i_r,n_r}. \nonumber
\end{align}
For the third sum \eqref{eq:lq_deriv3}, we observe that it is equal to
\begin{align*}
&\sum_{s=1}^r \sum_{t=1}^d\sum_{\substack{l_1+\cdots+l_d+l=k_1+\cdots+k_d+1 \\ p_1+\cdots+p_d+p=n_s}} (-1)^l\binom{k_1-1}{l_1-1}\cdots \binom{k_t-1}{l_t-2}\cdots \binom{k_d-1}{l_d-1} \binom{n_s}{p_1,\ldots,p_d,p} \\
&D_{i_1,n_1}\cdots D_{i_{s-1},n_{s-1}}[D_{i_s+l,p},D_{l_1,m_1+p_1}\cdots D_{l_t,m_t+1+p_1} \cdots D_{l_d,m_d+p_d}]D_{i_{s+1},n_{s+1}}\cdots D_{i_r,n_r}\\
&=\sum_{t=1}^d\sum_{s=1}^r\sum_{\substack{l_1+\cdots+l_d+l=k_1+\cdots+k_d+1 \\ p_1+\cdots+p_d+p=n_s}} (-1)^l\binom{k_1-1}{l_1-1}\cdots \binom{k_t}{l_t-1}\cdots \binom{k_d-1}{l_d-1} \binom{n_s}{p_1,\ldots,p_d,p} \\
&D_{i_1,n_1}\cdots D_{i_{s-1},n_{s-1}}[D_{i_s+l,p},D_{l_1,m_1+p_1}\cdots D_{l_t,m_t+1+p_1} \cdots D_{l_d,m_d+p_d}]D_{i_{s+1},n_{s+1}}\cdots D_{i_r,n_r} \\
&-\sum_{s=1}^r \sum_{t=1}^d\sum_{\substack{l_1+\cdots+l_d+l=k_1+\cdots+k_d \\ p_1+\cdots+p_d+p=n_s}} (-1)^{l+1}\binom{k_1-1}{l_1-1}\cdots \binom{k_t-1}{l_t-1}\cdots \binom{k_d-1}{l_d-1} \binom{n_s}{p_1,\ldots,p_d,p} \\
&D_{i_1,n_1}\cdots D_{i_{s-1},n_{s-1}}[D_{i_s+l+1,p},D_{l_1,m_1+p_1}\cdots D_{l_t,m_t+1+p_1} \cdots D_{l_d,m_d+p_d}]D_{i_{s+1},n_{s+1}}\cdots D_{i_r,n_r} \\
&=\der{\delta(f)}(g)+\sum_{s=1}^r \sum_{t=1}^dD_{i_1,n_1}\cdots D_{i_{s-1},n_{s-1}} \der{D_{k_1,m_1}\cdots D_{k_t,m_t+1}\cdots D_{k_d,m_d}}(D_{i_s+1,n_s} )D_{i_{s+1},n_{s+1}}\cdots D_{i_r,n_r}.
\end{align*}
Next, we observe that the first and second sum \eqref{eq:lq_deriv1} and \eqref{eq:lq_deriv2} are equal to
\begin{align*}
&\sum_{\substack{s,t=1 \\ s\neq t}}^r D_{i_1,n_1}\cdots \der{f}(D_{i_s,n_s})\cdots \delta(D_{i_t,n_t})\cdots D_{i_r,n_r} \\
&+\sum_{s=1}^r\sum_{\substack{l_1+\cdots+l_d+l=k_1+\cdots+k_d \\ p_1+\cdots+p_d+p=n_s+1}} (-1)^l\binom{k_1-1}{l_1-1}\cdots \binom{k_d-1}{l_d-1} \binom{n_s}{p_1,\ldots,p_d,p-1}  \\
&\hspace{2.4cm} D_{i_1,n_1}\cdots D_{i_{s-1},n_{s-1}}[D_{i_s+1+l,p},D_{l_1,m_1+p_1}\cdots D_{l_d,m_d+p_d}]D_{i_{s+1},n_{s+1}}\cdots D_{i_r,n_r}\big)\\
&=\sum_{\substack{s,t=1 \\ s\neq t}}^r D_{i_1,n_1}\cdots \der{f}(D_{i_s,n_s})\cdots \delta(D_{i_t,n_t})\cdots D_{i_r,n_r} \\
&+\sum_{s=1}^r\sum_{\substack{l_1+\cdots+l_d+l=k_1+\cdots+k_d \\ p_1+\cdots+p_d+p=n_s+1}} (-1)^l\binom{k_1-1}{l_1-1}\cdots \binom{k_d-1}{l_d-1} \binom{n_s+1}{p_1,\ldots,p_d,p}  \\
&\hspace{2.4cm} D_{i_1,n_1}\cdots D_{i_{s-1},n_{s-1}}[D_{i_s+1+l,p},D_{l_1,m_1+p_1}\cdots D_{l_d,m_d+p_d}]D_{i_{s+1},n_{s+1}}\cdots D_{i_r,n_r}\big)\\
&-\sum_{s=1}^r \sum_{t=1}^d\sum_{\substack{l_1+\cdots+l_d+l=k_1+\cdots+k_d \\ p_1+\cdots+p_d+p=n_s}} (-1)^l\binom{k_1-1}{l_1-1}\cdots \binom{k_d-1}{l_d-1} \binom{n_s}{p_1,\ldots,p_d,p}\\ 
&\hspace{1.4cm} D_{i_1,n_1}\cdots D_{i_{s-1},n_{s-1}}[D_{i_s+1+l,p},D_{l_1,m_1+p_1}\cdots D_{l_t,m_t+1+p_t}\cdots D_{l_d,m_d+p_d}]D_{i_{s+1},n_{s+1}}\cdots D_{i_r,n_r}\big)\\
&=\der{f}(\delta(g))-\sum_{s=1}^r \sum_{t=1}^d D_{i_1,n_1}\cdots D_{i_{s-1},n_{s-1}}\der{D_{k_1,m_1}\cdots D_{k_t,m_t+1}\cdots D_{k_d,m_d}}(D_{i_s+1,n_s})D_{i_{s+1},n_{s+1}}\cdots D_{i_r,n_r}.
\end{align*}
Both calculations together show that
\begin{align} \label{eq:lq_deriv_deriv}
\delta\big(\der{f}(g)\big)=\der{\delta(f)}(g)+\der{f}(\delta(g)).
\end{align}
Since $\delta$ is a derivation, we have
\begin{align} \label{eq:lq_deriv_comm}
\delta\big([f,g]\big)=[\delta(f),g]+[f,\delta(g)].
\end{align}
From \eqref{eq:lq_deriv_deriv} and \eqref{eq:lq_deriv_comm}, we derive the claimed identity \eqref{eq:lq_deriv_liebrack}.
\end{proof}

\section{The linearized double shuffle Lie algebra $\ls$ and the Lie algebra $\lqq$} \label{sec:ls_and_lq}

We first review  briefly the linearized double shuffle Lie algebra $\ls$ associated to the algebra $\Z$ of multiple zeta values. For more details we refer to \cite{Bro21}, \cite{IKZ06}, \cite{Sc15}.

\vspace{0.3cm}
Consider the alphabet $\X=\{x_0,x_1\}$, and denote by $\QX$ the free non-commutative algebra generated by $\X$. With the shuffle coproduct \begin{align} \label{eq:shco_QX}
\shco:\QX\to \QX\otimes\QX,\ x_i\mapsto x_i\otimes 1 +1 \otimes x_i, \qquad i\in\{0,1\},
\end{align}
we get a Hopf algebra structure $(\QX,\cdot,\shco)$. For a word $w=x_{s_1}\cdots x_{s_k}\in \QX$ define the \emph{weight} and \emph{depth} as
\begin{align*}
\wt(w)=k,\qquad \dep(w)=\#\{i\mid s_i=1\}.
\end{align*}
These two notions extend to gradings on the Hopf algebra $(\QX,\cdot,\shco)$.

Next, consider the alphabet $\Y=\{y_1,y_2,\ldots\}$, and similarly denote by $\QY$ the free non-commutative algebra generated by $\Y$. Also in this case we get a Hopf algebra structure $(\QY,\cdot,\shco)$ where the shuffle coproduct is given as before by 
\begin{align} \label{eq:shco_QY}
\shco(y_i)=y_i\otimes 1 +1\otimes y_i,\qquad i\geq1.
\end{align}

We have a canonical projection from $\QX$ to $\QY$ given by
\begin{align} \label{eq:piY}
\piY:\QX\to \QY,\quad  x_0^{k_1-1}x_1\cdots x_0^{k_d-1}x_1x_0^{n}\mapsto\begin{cases} y_{k_1}\cdots y_{k_d}, &\quad n=0,\\ 0 &\quad \text{else}. \end{cases}
\end{align}

\begin{defi} \label{def:ls} Let $\ls$ be the space of all $\phi\in \QX$ such that
\begin{flalign*}
\qquad \text{(i)} \quad & (\phi\mid x_0)=(\phi\mid x_1)=0, &&\\
\qquad \text{(ii)} \quad & (\phi\mid x_0^{m-1}x_1)=0, \qquad m \text{ even}, &&\\
\qquad \text{(iii)} \quad & \shco\phi=\phi\otimes 1+1\otimes \phi, &&\\
\qquad \text{(iv)} \quad & \shco\piY(\phi)=\piY(\phi)\otimes 1+1\otimes\piY(\phi).
\end{flalign*}
\end{defi}

These conditions are equivalent to the extended double shuffle relations in the algebra $\Z$ of multiple zeta values modulo products, lower depths, and $\zeta(2)$.
The space $\ls$ is bi-graded with respect to the weight and depth
\[\ls=\bigoplus_{k,d\geq0} \ls_{k,d}.\]
Similar to Theorem \ref{thm:lq_parity}, we have the following.

\begin{thm} \label{thm:ls_parity} \cite[Theorem 7]{IKZ06} If $k\not\equiv d\mod 2$, then we have
\[\ls_{k,d}=\{0\}.\]
\end{thm}

\begin{defi} \label{def:Ihara_bracket} For any $w\in \QX$, let $d_w:\QX\to \QX$ be the derivation given by
\[d_w(x_0)=0,\quad d_w(x_1)=[x_1,w].\]
The \emph{Ihara bracket} on $\QX$ is defined by
\[\{f,g\}=d_f(g)-d_g(f)+[f,g],\qquad f,g\in \QX.\]
\end{defi}

\begin{thm} \cite[Theorem 5.5]{Bro21} The pair $(\ls,\{-,-\})$ is a weight- and depth-graded Lie algebra.
\end{thm}

For the double shuffle Lie algebra $(\dm,\{-,-\})$ introduced and studied in \cite{Ra00}, we expect to have an isomorphism
\[\Z/(\zeta(2))\overset{?}{\simeq} \mathcal{U}(\dm)^\vee.\] 
By \cite{Ra00}, this isomorphism is implied by the main conjecture for multiple zeta values that all algebraic relations are induced by the extended double shuffle relations.
Moreover by construction, the associated depth-graded $\grD \dm$ embeds into the linearized double shuffle Lie algebra $\ls$. In contrast to the case of multiple zeta values, see Remark \ref{rem:lq_grDZ_q_not_inj}, it is conjectured in \cite{IKZ06} that this morphism is an isomorphism
\begin{align} \label{eq:grDdm0_ls_iso}
\grD \dm \overset{?}{\simeq} \ls.
\end{align}
Therefore, the following is expected for the associated depth-graded algebra $\grDZ$.

\begin{con} \label{conj:mzv_depthgraded_iso} There is an algebra isomorphism
\[\grDZ/(\zeta(2))\simeq \mathcal{U}(\ls)^\vee.\]
\end{con}

In the remaining of this section, we explain the connection of the linearized double shuffle Lie algebra $(\ls,\{-,-\})$ and the Lie algebra $(\lqq,\ariemp)$ introduced in Section \ref{sec:def_lq} and \ref{sec:lq_Lie_alg}.

Define the two $\QQ$-linear maps $\theta_\X,\theta_\Y$ by
\begin{align}
\label{eq:thetaX}
\thetaX:&\ \QX\hookrightarrow \QB, \quad x_{s_1}\cdots x_{s_k} \mapsto b_{s_1}\cdots b_{s_k}, \\
\label{eq:thetaY}
\thetaY:&\ \QY\hookrightarrow \QB, \quad \hspace{0.08cm} y_{k_1}\cdots y_{k_d} \mapsto b_{k_d}\cdots b_{k_1}.
\end{align}
Evidently, both maps are coalgebra morphisms for the shuffle coproduct $\shco$. Moreover, the maps $\thetaX,\ \thetaY$ are compatible with the projections $\pizero$, $\piY$ occurring in Definition \ref{def:lq}, \ref{def:ls}.

\begin{lem} \label{lem:thetas_proj} We have
\begin{enumerate}[(i)]
\item $\tau\circ \pizero\circ \thetaX=\thetaY\circ \piY$,
\item $\pizero\circ \thetaX=\tau\circ\thetaY\circ \piY$.
\end{enumerate}
\end{lem}

\begin{proof} We first prove (i). Let $w\in \QX$. If $w=\widetilde{w}x_0$ ends with $x_0$, then
\begin{align*}
\big(\tau\circ \pizero\circ \thetaX\big)(w)=\big(\tau\circ \pizero\big)\big(\thetaX(\widetilde{w})b_0\big) =0 = \piY(w) =\big(\thetaY\circ \piY\big)(w).
\end{align*}
If $w$ does not end with $x_0$, we can write $w=x_0^{k_1-1}x_1\cdots x_0^{k_d-1}x_1$ for some $k_1,\ldots,k_d\geq1$. We compute
\begin{align*}
\big(\tau\circ \pizero\circ \thetaX\big)(w)=\big(\tau\circ \pizero\big)(b_0^{k_1-1}b_1\cdots b_0^{k_d-1}b_1)=\tau(b_0^{k_1-1}b_1\cdots b_0^{k_d-1}b_1)=b_{k_d}\cdots b_{k_1},
\end{align*}
and on the other hand
\begin{align*}
\big(\thetaY\circ\piY)(w)=\thetaY(y_{k_1}\cdots y_{k_d})=b_{k_d}\cdots b_{k_1}.
\end{align*}
Altogether, the maps $\tau\circ \pizero\circ \thetaX$ and $\thetaY\circ \piY$ agree on $\QX$. The equality in (ii) follows now from applying $\tau$ to both sides of the equation in (i), as $\tau$ is an involution.
\end{proof}

Next, we study the compatibility of the maps $\thetaX$, $\thetaY$ and the Lie brackets $\{-,-\}$, $\ariemp$.

\begin{prop} \label{prop:thetaX_Lie_bracks} We have for $v,w\in \QX$
\[ \ari{\thetaX(v)}{\thetaX(w)}=\thetaX(\{v,w\}).\]
\end{prop}

\begin{proof} By Definition \ref{def:Lie_bracket_lq}, we have for a word $w$ just containing the letters $b_0,b_1$ that $\der{w}(b_0)=0$ and $\der{w}(b_i)=[b_i,w]$ for $i\geq1$. This is similar to Definition \ref{def:Ihara_bracket}, so we have \[\der{\thetaX(w)}(\thetaX(v))=\thetaX(d_w(v))\] 
for all $v,w\in \QX$. Since the Ihara bracket $\{-,-\}$ and the Lie bracket $\ariemp$ are defined in the same way (cf Definition \ref{def:Lie_bracket_lq}, \ref{def:Ihara_bracket}), we deduce
\[\ari{\thetaX(v)}{\thetaX(w)}=\thetaX(\{v,w\}),\qquad v,w\in \QX. \qedhere\]
\end{proof}

\begin{prop}\label{prop:thetaY_Lie_bracks} For $\phi,\psi\in \ls$, we have
\[\ari{\thetaY(\piY(\phi))}{\thetaY(\piY(\psi))}=\thetaY\big(\piY(\{\phi,\psi\})\big).\]
\end{prop}

\begin{proof} 
Let $\phi\in \ls$. We first show that $\pizero(\thetaX(\phi))$ is $\rho$-invariant, where the map $\rho$ is given in Definition \ref{def:rho}. By the Definition \ref{def:ls} of $\ls$, $\phi$ is primitive for the shuffle coproduct. Since $\thetaX$ is a coalgebra morphism, also $\thetaX(\phi)$ is primitive for the shuffle coproduct. Hence, we get
$S(\thetaX(\phi))=-\thetaX(\phi)$, where $S$ is the antipode of the Hopf algebra $(\QB,\cdot,\shco)$ as in \eqref{eq:antipode}. Recall that $S_0=\pizero\circ S\circ \sec$. By Proposition \ref{prop:pi0_sec_inverse}, the maps $\sec$ and $\pizero$ are inverse on $\QDbi$, and we have $\thetaX(\phi)\in \QDbi$. Therefore, we get  
\begin{align} \label{eq:thetaXphi_S0} S_0\big(\pizero(\thetaX(\phi))\big)=-\pizero(\thetaX(\phi)).
\end{align}
Similarly, $\thetaY$ is a coalgebra morphism, so as $\piY(\phi)$ is primitive, also $\thetaY(\piY(\phi))$ is primitive for the shuffle coproduct. Hence, we get $S\big(\thetaY(\piY(\phi))\big)=-\thetaY(\piY(\phi))$. As $\thetaY(\piY(\phi))$ does not contain the letter $b_0$, we have $S\big(\thetaY(\piY(\phi))\big)=S_0\big(\thetaY(\piY(\phi))\big)$, and hence
\begin{align} \label{eq:thetaYphi_S0}
S_0\big(\thetaY(\piY(\phi))\big)=-\thetaY(\piY(\phi)).
\end{align}
We deduce from \eqref{eq:thetaXphi_S0} and \eqref{eq:thetaYphi_S0} with Lemma \ref{lem:thetas_proj}
\begin{align*}
(S_0\circ \tau\circ S_0\circ \tau)\big(\pizero(\thetaX(\phi))\big)&=(S_0\circ \tau\circ S_0)\big(\thetaY(\piY(\phi))\big) \\
&=-(S_0\circ \tau)\big(\thetaY(\piY(\phi))\big) \\
&=-S_0\big(\pizero(\thetaX(\phi))\big) \\
&=\pizero(\thetaX(\phi)).
\end{align*}
Without loss of generality we can assume that $\phi$ is homogeneous in weight and depth, as the space $\ls$ is bigraded. Thus, with the calculation in \eqref{eq:rho_equal_S_tau}, we get
\begin{align*}
(-1)^{\wt(\phi)+\dep(\phi)}\rho\big(\pizero(\thetaX(\phi))\big)=\pizero(\thetaX(\phi)).
\end{align*}
By Theorem \ref{thm:ls_parity}, the element $\phi$ is only nonzero if $\wt(\phi)\equiv \dep(\phi) \mod 2$. We deduce
\[\rho\big(\pizero(\thetaX(\phi))\big)=\pizero(\thetaX(\phi)).\]
So, let $\phi,\psi\in\ls$. Then, we compute
\begin{align*}
\ari{\thetaY(\piY(\phi))}{\thetaY(\piY(\psi))}&=\arizero{\thetaY(\piY(\phi))}{\thetaY(\piY(\psi))} \\
&\overset{\ref{lem:thetas_proj}}{=}\arizero{(\tau\circ\pizero\circ\thetaX)(\phi)}{(\tau\circ\pizero\circ\thetaX)(\psi)} \\
&\overset{\ref{cor:Lie_brack_tau_inv}}{=}\tau\big(\arizero{(\pizero\circ\thetaX)(\phi)}{(\pizero\circ\thetaX)(\psi)} \big) \\
&\overset{\ref{eq:Lie_brack_QB0}}{=} (\tau\circ \pizero)\big(\ari{\thetaX(\phi)}{\thetaX(\psi)} \big) \\
&\overset{\ref{prop:thetaX_Lie_bracks}}{=}(\tau\circ \pizero\circ \thetaX)(\{\phi,\psi\}) \\
&\overset{\ref{lem:thetas_proj}}{=} \thetaY(\piY(\{\phi,\psi\})).
\qedhere\end{align*}
\end{proof}

\begin{prop} \label{prop:thetaXY_ari}For $\phi,\psi\in\ls$, we have
\[\ari{\thetaX(\phi)}{\thetaY(\piY(\psi))}=0.\] 
\end{prop}

\begin{proof} Let $w\in \QQ\langle b_0,b_1\rangle$ be any polynomial in the letters $b_0,b_1$ and $v=b_{k_1}\cdots b_{k_d}$ with $k_1,\ldots,k_d\geq1$. Then, we have by Definition \ref{def:Lie_bracket_lq}
\begin{align*}
\der{w}(v)=\sum_{i=1}^d b_{k_1}\cdots b_{k_{i-1}}[b_{k_i},w]b_{k_{i+1}}\cdots b_{k_d} =b_{k_1}\cdots b_{k_d}w-wb_{k_1}\cdots b_{k_d} =-[w,v].
\end{align*}
Since $\thetaX(\phi)\in \QQ\langle b_0,b_1\rangle$ and $\thetaY(\piY(\psi))\in\QQ\langle b_i\mid i\geq1\rangle$ for $\phi,\psi\in\ls$, we get
\begin{align} \label{eq:derA_on_b0b1_bi}
\ari{\thetaX(\phi)}{\thetaY(\piY(\psi))}&=\der{\thetaX(\phi)}(\thetaY(\piY(\psi)))-\der{\thetaY(\piY(\psi))}(\thetaX(\phi))+[\thetaX(\Phi),\thetaY(\piY(\psi))] \\
&=-\der{\thetaY(\piY(\psi))}(\thetaX(\phi)). \nonumber
\end{align}
Thus, we are left with showing that $\der{\thetaY(\piY(\psi))}(\thetaX(\phi))=0$. Recall from the proof of Proposition \ref{prop:thetaY_Lie_bracks} that $\pizero(\thetaX(\psi))$ is $\rho$-invariant. Thus, we compute
\begin{align*}
\pizero\big(\der{\thetaY(\piY(\psi))}(\thetaX(\phi))\big)&\overset{\ref{lem:thetas_proj}}{=}\derzero{(\tau\circ \pizero\circ \thetaX)(\psi)}((\tau\circ\thetaY\circ \piY)(\phi)) \\
&=\derrzero{(\tau\circ \pizero\circ \thetaX)(\psi)}((\tau\circ\thetaY\circ \piY)(\phi))-\derlzero{(\tau\circ \pizero\circ \thetaX)(\psi)}((\tau\circ\thetaY\circ \piY)(\phi)) \\
&\overset{\ref{prop:dA0_and_tau}}{=} \tau\Big(\derrzero{(\pizero\circ\thetaX)(\psi)}((\thetaY\circ\piY)(\phi))\Big) + \tau\Big((\pizero\circ\thetaX)(\psi)\conczero(\thetaY\circ\piY)(\phi)\Big) \\
& -(\tau\circ \pizero\circ \thetaX)(\psi)\conczero(\tau\circ\thetaY\circ\piY)(\phi)-\tau\Big(\derlzero{(\rho\circ \pizero\circ\thetaX)(\psi)}((\thetaY\circ\piY)(\phi))\Big) \\
&\overset{\ref{cor:Lie_brack_tau_inv},\ref{lem:thetas_proj}}{=} \tau\Big(\derzero{(\pizero\circ\thetaX)(\psi)}((\thetaY\circ\piY)(\phi))\Big)+ \tau\Big((\pizero\circ\thetaX)(\psi)\conczero(\thetaY\circ\piY)(\phi)\Big) \\
&\hspace{0.4cm}-(\thetaY\circ\piY)(\psi)\conczero(\pizero\circ \thetaX)(\phi) \\
&=(\tau\circ\pizero)\Big(\der{\thetaX(\psi)}((\thetaY\circ\piY)(\phi))+\thetaX(\psi)(\thetaY\circ\piY)(\phi)\Big)\\
&\hspace{0.4cm}-\pizero\Big((\thetaY\circ\piY)(\psi)\thetaX(\phi)\Big) \\
&\overset{\ref{eq:derA_on_b0b1_bi}}{=} (\tau\circ \pizero)\Big((\thetaY\circ\piY)(\phi)\thetaX(\psi)\Big)-\pizero\Big((\thetaY\circ\piY)(\psi)\thetaX(\phi)\Big) \\
&=(\tau\circ\pizero\circ\thetaX)(\psi)(\tau\circ\thetaY\circ\piY)(\phi)-(\thetaY\circ\piY)(\psi)(\pizero\circ \thetaX)(\phi) \\
&\overset{\ref{lem:thetas_proj}}{=}0.
\end{align*}
We deduce $\der{\thetaY(\piY(\psi))}(\thetaX(\phi))=0$ and hence get the claimed formula.
\end{proof}

As part (iii) of Theorem \ref{main_thm}, we have the following.

\begin{thm} \label{thm:emb_ls_lq} There is an injective Lie algebra morphism
\[\theta:(\ls,\{-,-\})\hookrightarrow (\lqq,\ariemp), \quad \phi\mapsto \thetaX(\phi)+\thetaY(\piY(\phi)).\]
\end{thm}

\begin{proof}
We first show that $\theta(\phi)\in \lqq$ for $\phi\in\ls$, so that the map $\theta$ is well-defined. Since $\phi\in \ls$ satisfies $(\phi\mid x_0)=0$ and $(\phi\mid x_0^{m-1}x_1)=0$ for $m$ even, we deduce
\[(\thetaX(\phi)\mid b_0)=(\thetaX(\phi)\mid b_0^{m-1}b_1)=0,\qquad (\thetaY(\piY(\phi))\mid y_m)=0, \qquad m \text{ even}.\]
Moreover, by construction $\theta(\phi)$ never contains a word in which $b_0$ and $b_i$ for some $i>1$ occur at the same time. Thus, we also have $(\theta(\phi)\mid b_0^mb_k)=0$ for $k+m$ even and $k>1$, $m>0$. Altogether, this shows that $\theta(\phi)$ satisfies the conditions (i), (ii) in the Definition \ref{def:lq} of $\lqq$.

Next, we show that $\theta(\phi)$ also satisfies condition (iii), i.e., that $\theta(\phi)$ is primitive von $\shco$. By definition of $\ls$, $\phi\in \QX$ and $\piY(\phi)\in \QY$ are primitive for the shuffle coproducts on $\QX$ and $\QY$. Since $\thetaX$, $\thetaY$ are coalgebra morphisms, we deduce
\begin{align*}
\shco\big(\theta(\phi)\big)&=\shco\big(\thetaX(\phi)\big)+\shco\big(\thetaY(\piY(\phi))\big) \\
&=\thetaX\big(\shco(\phi)\big)+\thetaY\big(\shco(\piY(\phi))\big) \\
&=\thetaX\big(\phi\otimes 1+1\otimes \phi\big)+\thetaY\big(\piY(\phi)\otimes1+1\otimes\piY(\phi)\big)\\
&=\big(\thetaX(\phi)+\thetaY(\piY(\phi))\big)\otimes1+1\otimes\big(\thetaX(\phi)+\thetaY(\piY(\phi))\big) \\
&=\theta(\phi)\otimes1 +1\otimes\theta(\phi).
\end{align*}
Finally, we show that $\theta(\phi)$ satisfies condition (iv) in the Definition \ref{def:lq} of $\lqq$, i.e., that $\pizero(\theta(\phi))$ is $\tau$-invariant. By Lemma \ref{lem:thetas_proj} and $\tau$ being an involution, we have
\begin{align*}
\tau\big(\pizero(\theta(\phi))\big)&=\tau\big(\pizero(\thetaX(\phi))\big)+\tau\big(\thetaY(\piY(\phi))\big) \\
&=\thetaY(\piY(\phi))+\pizero(\thetaX(\phi)) \\
&=\pizero\big(\thetaY(\piY(\phi))+\thetaX(\phi) \big) \\
&=\pizero(\theta(\phi)).
\end{align*}
Here, we used that by definition $\thetaY(\piY(\phi))$ does not contain the letter $b_0$ and hence we have $\pizero\big(\thetaY(\piY(\phi))\big)=\thetaY(\piY(\phi))$.

Moreover, $\theta$ is a Lie algebra morphism, as we have for $\phi,\psi\in\ls$
\begin{align*}
\ari{\theta(\phi)}{\theta(\psi)}&\overset{\ref{prop:thetaXY_ari}}{=}\ari{\thetaX(\phi)}{\thetaX(\psi)}+\ari{\thetaY(\piY(\phi))}{\thetaY(\piY(\psi))}\\
&\overset{\ref{prop:thetaX_Lie_bracks},\ref{prop:thetaY_Lie_bracks}}{=} \thetaX(\{\phi,\psi\})+\thetaY(\piY(\{\phi,\psi\})) \\
&=\theta(\{\phi,\psi\}).
\qedhere \end{align*}
\end{proof}

\section{The Lie algebra $\lqq$ and bimoulds} \label{sec:bimoulds}

In this section, we translate our results from the previous sections into the language of bimoulds. We just introduce the needed notions for bimoulds, for more details and background we refer to \cite{Ec11}, \cite{Sc15}.

\vspace{0.3cm} 
A \textit{bimould} is a sequence of (commutative) polynomials
\begin{align*} A=\left(A_d\binom{X_1,\ldots,X_d}{Y_1,\ldots,Y_d}\right)_{d\geq0}&=\left(A_0(\emptyset),A_1\binom{X_1}{Y_1},A_2\binom{X_1,X_2}{Y_1,Y_2},\ldots\right)\\
&\in\bigoplus_{d\geq0}\QQ[X_1,Y_1,\ldots,X_d,Y_d],
\end{align*}
where we only allow finitely many components $A_d$ to be nonzero. The space of all bimoulds together with the multiplication
\begin{align*}
\operatorname{mu}(A,B)_d\binom{X_1,\ldots,X_d}{Y_1,\ldots,Y_d}=\sum_{i=0}^d A_i\binom{X_1,\ldots,X_i}{Y_1,\ldots,Y_i}B_{d-i}\binom{X_{i+1},\ldots,X_d}{Y_{i+1},\ldots,Y_d}
\end{align*}
is a $\QQ$-algebra.

Consider the alphabet $\Dbi=\{D_{k,m}\mid k\geq 1, m\geq0\}$ from \eqref{eq:def_Dbi}. The \textit{coefficient map} $\varphi_A:\QDbi\to \QQ$ of the bimould $A$ is the $\QQ$-linear map defined by
\begin{align*}	A_d\binom{X_1,\ldots,X_d}{Y_1,\ldots,Y_d}=\sum_{\substack{k_1,\ldots,k_d\geq1 \\ m_1,\ldots,m_d\geq0}} \varphi_A(D_{k_1,m_1}\cdots D_{k_d,m_d})X_1^{k_1-1}Y_1^{m_1}\cdots X_d^{k_d-1}Y_d^{m_d}, \quad d\geq0.
\end{align*}

\begin{defi} (i) A bimould $A$ is called \emph{alternal} if its coefficient map $\varphi_A$ satisfies
\[\varphi_A(v\shuffle w)=0 \qquad \text{for all } v,w \in \QDbi\backslash \QQ.\]
(ii) A bimould $A$ is called \emph{swap invariant} if
\[A_d\binom{X_1,\ldots,X_d}{Y_1,\ldots,Y_d}=A_d\binom{Y_d,Y_{d-1}+Y_d,\ldots,Y_1+\cdots+Y_d}{X_d-X_{d-1},\ldots,X_2-X_1,X_1}, \qquad d\geq1.\]
(iii) The set of all alternal, swap invariant bimoulds $A$, such that $A_0=0$ and $A_1$ is even, is denoted by $\BARI$.
\end{defi}

Note that the swap invariance defined in (ii) slightly differs from the one in \cite[§2.4]{Sc15}, the connection is given by pre- and post-composing with the substitution $\binom{X_1,\ldots,X_d}{Y_1,\ldots,Y_d}\mapsto\binom{Y_d,\ldots,Y_1}{X_d,\ldots,Y_1}$.

\begin{prop} \label{prop:bari_parity} \cite{Sc15} Let $A\in \BARI$. Then, the components $A_d\binom{X_1,\ldots,X_d}{Y_1,\ldots,Y_d}$ are even polynomials for all $d\geq1$.
\end{prop}

\begin{proof}
If $A\in\BARI$, then $A$ is by definition an alternal bimould and also $\operatorname{swap}(A)$ is alternal 5. Thus, by \cite[Lemma 2.5.5]{Sc15}\footnote{Though \cite[Lemma 2.5.5]{Sc15} is only stated for moulds, the proof is actually done for bimoulds. Moreover, $\binom{X_1,\ldots,X_d}{Y_1,\ldots,Y_d}\mapsto\binom{Y_d,\ldots,Y_1}{X_d,\ldots,Y_1}$ is an automorphism on alternal bimoulds, so also with the slight modification of swap invariance we can still apply \cite[Lemma 2.5.5.]{Sc15}.} the bimould $A$ is invariant under the substitution $A_d\binom{X_1,\ldots,X_d}{Y_1,\ldots,Y_d}\mapsto A_d\binom{-X_1,\ldots,-X_d}{-Y_1,\ldots,-Y_d}$. This means, the component $A_d$ is an even polynomial for each $d\geq1$.
\end{proof}

We construct an isomorphism from the space $\lqq$ introduced in Section $\ref{sec:def_lq}$ to $\BARI$. We use the description of $\lqq$ in the alphabet $\Dbi$ explained in Lemma \ref{lem:lqq_in_Dbi}.
By the \emph{depth} of a word $w\in \QDbi$ we mean the number of its letters. We write $\QDbi_d$ for the homogeneous subspace of $\QDbi$ of depth $d$, and for $f\in \QDbi$ we denote by $f_d$ its projection to $\QDbi_d$.

\begin{defi} \label{def:bimap} For each $d\geq1$, define the isomorphism 
\begin{align*}
\bimap_d: \QDbi_d&\to \QQ[X_1,Y_1,\ldots,X_d,Y_d], \\
D_{k_1,m_1}\cdots D_{k_d,m_d}&\mapsto X_1^{k_1-1}Y_1^{m_1}\cdots X_d^{k_d-1}Y_d^{m_d}.
\end{align*}
Then, for any $f\in \QDbi$, we define the bimould $\bimap(f)$ by $\bimap(f)_0=0$ and
\[\bimap(f)_d\binom{X_1,\ldots,X_d}{Y_1,\ldots,Y_d}=\bimap_d(f_d),\qquad d\geq1.\]
\end{defi}

\begin{prop} \label{prop:bimap_vsiso} We have an isomorphism of vector spaces
\[\lqq\to \BARI,\quad \Phi\mapsto \bimap(\Phi).\]
\end{prop}

\begin{proof} Similar to \eqref{eq:primitive_zero_on_prod}, an element $\Phi$ lies in $\LieDbi$ if and only if the map \[\operatorname{coef}_\Phi:\QDbi\to\QQ,\ w\mapsto (\Phi\mid w)\] 
satisfies 
\[\operatorname{coef}_\Phi(v\shuffle w)=0\quad  \text{ for all } v,w \in \QDbi\backslash\QQ.\]  
Since $\operatorname{coef}_\Phi$ is also the coefficient map of the bimould $\bimap(\Phi)$, we deduce the equivalence of being a Lie element and alternality under the map $\bimap$.
Condition (ii') in the alternative description of $\lqq$ given in Lemma \ref{lem:lqq_in_Dbi} is equivalent to $\bimap(\Phi)_1$ being an even polynomial. 
So finally, we have to consider the $\tau_{\Dbi}$-invariance given in condition (iv') of Lemma \ref{lem:lqq_in_Dbi}. Recall that we have
\begin{align*}
&\tau_{\Dbi}(D_{k_1,m_1}\cdots D_{k_d,m_d})=\sum_{\substack{l_1^{(s)}+\cdots+l_s^{(s)}=k_s-1,\ s=1,\ldots,d \\ n_1+\cdots+n_{d-1}+n=m_1+\cdots+m_d}} \prod_{s=1}^d \binom{k_s-1}{l_1^{(s)},\ldots,l_s^{(s)}} \binom{m_s}{n_s}(-1)^{m_s+n_s} \\
&\cdot D_{m_d+m_{d-1}-n_{d-1}+1,l_d^{(d)}}D_{n_{d-1}+m_{d-2}-n_{d-2}+1,l_{d-1}^{(d-1)}+l_{d-1}^{(d)}}\cdots D_{n_2+m_1-n_1+1,l_2^{(2)}+\cdots+l_2^{(d)}}D_{n_1+1,l_1^{(1)}+\cdots l_1^{(d)}},
\end{align*}
where $n_d:=m_d$. Thus, we deduce from the definition of $\bimap$ that
\begin{align*}
&\bimap\big(\tau_{\Dbi}(D_{k_1,m_1}\cdots D_{k_d,m_d})\big)_d\binom{X_1,\ldots,X_d}{Y_1,\ldots,Y_d}\\
&\hspace{6cm}=\bimap(D_{k_1,m_1}\cdots D_{k_d,m_d})_d\binom{Y_d,Y_{d-1}+Y_d,\ldots,Y_1+\cdots+Y_d}{X_d-X_{d-1},\ldots,X_2-X_1,X_1}.
\end{align*}
Since $\bimap$ is $\QQ$-linear, we get for all $f\in\QDbi$
\[\bimap\big(\tau_{\Dbi}(f)\big)=\operatorname{swap}(\bimap(f)).\]
Therefore, $\tau_{\Dbi}$-invariance of $\Phi\in\lqq$ is equivalent to $\operatorname{swap}$-invariance of $\bimap(\Phi)$.
\end{proof}

We compare the Lie algebra structures on $\lqq$ and $\BARI$ via the map $\bimap$. First, we introduce the Lie algebra structure on $\BARI$ from \cite{Ec11}, \cite{Sc15}.

\begin{defi} \label{def ari} For bimoulds $A,B$, define the bimould $\operatorname{arit}_B(A)$ by
\begin{align*}
\operatorname{arit}_B(A)_d&\binom{X_1,\ldots,X_d}{Y_1,\ldots,Y_d}\\
&=\sum_{j=1}^d\sum_{i=1}^{d-j} A_{d-j}\binom{X_1,\ldots,X_{i-1},X_i,X_{i+j+1},\ldots,X_d}{Y_1,\ldots,Y_{i-1},Y_i+\cdots+Y_{i+j},Y_{i+j+1},\ldots,Y_d} \\
&\hspace{8cm} \cdot B_j\binom{X_{i+1}-X_i,\ldots, X_{i+j}-X_i}{Y_{i+1},\ldots,Y_{i+j}} \\
&-\sum_{j=1}^d\sum_{i=1}^{d-j} A_{d-j}\binom{X_1,\ldots,X_{i-1},X_{i+j},X_{i+j+1},\ldots,X_d}{Y_1,\ldots,Y_{i-1},Y_i+\cdots+Y_{i+j},Y_{i+j+1},\ldots,Y_d} \\
&\hspace{7.8cm} \cdot B_j\binom{X_i-X_{i+j},\ldots, X_{i+j-1}-X_{i+j}}{Y_i,\ldots,Y_{i+j-1}}. 
\end{align*}
Then, we let $\operatorname{ari}(A,B)$ be given by
\begin{align} \label{eq:def_ari}
\operatorname{ari}(A,B)=\operatorname{arit}_A(B)-\operatorname{arit}_B(A)+\operatorname{mu}(A,B)-\operatorname{mu}(B,A).
\end{align}
\end{defi} 

\begin{prop} \label{prop:BARI_Liealg} \cite{Sc15} The pair $(\BARI,\operatorname{ari})$ is a Lie algebra.
\end{prop}

\begin{proof} By \cite[Proposition 2.5.2]{Sc15}\footnote{The results in \cite{Sc15} are formulated only for moulds, but the proofs are done for bimoulds.}, the ari bracket preserves the space of alternal bimoulds. Next, any $A\in\BARI$ is an alternal bimould and $\operatorname{swap}(A)$ is also alternal. So, we can apply \cite[Lemma 2.5.5]{Sc15} and obtain that any $A\in\BARI$ is a push invariant bimould. Thus, by \cite[Lemma 2.4.1]{Sc15} we get
\[\operatorname{swap}\big(\operatorname{ari}\big(\operatorname{swap}(A),\operatorname{swap}(B)\big)\big)=\operatorname{ari}(A,B), \qquad A,B\in\BARI.\]
So, $\operatorname{ari}(A,B)$ is also swap invariant for $A,B\in\BARI$.
\end{proof}

\begin{thm} \label{thm:lq_BARI_Lieiso} We have an isomorphism of Lie algebras
\[\big(\lqq,\ariemp\big)\to \big(\BARI,\operatorname{ari}\big),\quad \Phi\mapsto\bimap(\Phi).\]
\end{thm}

\begin{proof} By Proposition \ref{prop:bimap_vsiso}, the map is a vector space isomorphism. Therefore, it suffices to show that
\begin{align*}
\bimap\big(\ari{\Phi}{\Psi}\big)=\operatorname{ari}\big(\bimap(\Phi),\bimap(\Psi)\big),\qquad \Phi,\Psi\in\lqq.
\end{align*}
It is straight-forward to see that
\begin{align*}
\bimap(fg)=\operatorname{mu}(\bimap(f),\bimap(g)), \qquad f,g\in \QDbi.
\end{align*}
The Lie bracket \eqref{eq:def_liebracket_lq} on $\lqq$ and the ari bracket \eqref{eq:def_ari} on $\BARI$ are defined in a similar way, thus we show that
\begin{align*}
\bimap\big(\der{f}(g)\big)=\operatorname{arit}_{\bimap(f)}\big(\bimap(g)\big),\qquad f,g\in\QDbi.
\end{align*}
Since all maps are $\QQ$-linear, we can assume $f=D_{k_1,m_1}\cdots D_{k_d,m_d}$ and $g=D_{i_1,n_1}\cdots D_{i_r,n_r}$, so the bimoulds $\bimap(f)$ and $\bimap(g)$ are concentrated in depth $d$ and $r$. First, observe that for $1\leq j_1\leq j_2\leq r$
\begin{align} \label{eq:ad(b0)_and_Y}
&(Y_{j_1}+\cdots+Y_{j_2})^m\bimap(g)_r\binom{X_1,\ldots,X_r}{Y_1,\ldots,Y_r} \nonumber \\
&=\sum_{p_{j_1}+\cdots+p_{j_2}=m} \binom{m}{p_{j_1},\ldots,p_{j_2}} X_1^{i_1-1}Y_1^{n_1}\cdots X_{j_1}^{i_{j_1}-1}Y_{j_1}^{n_{j_1}+p_{j_1}}\cdots X_{j_2}^{i_{j_2}-1}Y_{j_2}^{n_{j_2}+p_{j_2}}\cdots X_r^{i_r-1}Y_r^{n_r} \nonumber \\
&=\bimap(D_{i_1,n_1}\cdots \ad(b_0)^m(D_{i_{j_1},n_{j_1}}\cdots D_{i_{j_2},n_{j_2}})\cdots D_{i_r,n_r})_r\binom{X_1,\ldots,X_r}{Y_1,\ldots,Y_r}.
\end{align}
Thus, we get
\begin{align*}
&\operatorname{arit}_{\bimap(f)}\big(\bimap(g)\big)_{d+r}\binom{X_1,\ldots,X_{d+r}}{Y_1,\ldots,Y_{d+r}}\\ 
&=\sum_{s=1}^r \big(X_1^{i_1-1}Y_1^{n_1}\cdots X_{s-1}^{i_{s-1}-1}Y_{s-1}^{n_{s-1}}X_s^{i_s-1}(Y_s+\cdots+Y_{d+s})^{n_s}X_{d+s+1}^{i_{s+1}-1}Y_{d+s+1}^{n_{s+1}}\cdots X_{d+r}^{i_r-1}Y_{d+r}^{n_r}\big)\\
&\hspace{1.5cm}\cdot\big((X_{s+1}-X_s)^{k_1-1}Y_{s+1}^{m_1}\cdots (X_{s+d}-X_s)^{k_d-1}Y_{s+d}^{m_d}\big) \\
&- \sum_{s=1}^r \big(X_1^{i_1-1}Y_1^{n_1}\cdots X_{s-1}^{i_{s-1}-1}Y_{s-1}^{n_{s-1}}X_{d+s}^{i_s-1}(Y_s+\cdots+Y_{d+s})^{n_s}X_{d+s+1}^{i_{s+1}-1}Y_{d+s+1}^{n_{s+1}}\cdots X_{d+r}^{i_r-1}Y_{d+r}^{n_r}\big)\\
&\hspace{1.5cm}\cdot\big((X_s-X_{d+s})^{k_1-1}Y_s^{m_1}\cdots (X_{d+s-1}-X_{d+s})^{k_d-1}Y_{d+s-1}^{m_d}\big) \\
&\overset{\eqref{eq:ad(b0)_and_Y}}{=}\sum_{s=1}^r \sum_{l_1=1}^{k_1}\cdots \sum_{l_d=1}^{k_d} (-1)^{k_1-l_1+\cdots+k_d-l_d} \binom{k_1-1}{l_1-1}\cdots \binom{k_d-1}{l_d-1} \\
&\hspace{0.4cm}\cdot\bimap\Big(D_{i_1,n_1}\cdots D_{i_{s-1},n_{s-1}}\ad(b_0)^{n_s}\big(D_{i_s+k_1-l_1+\cdots+k_d-l_d,0}D_{l_1,m_1}\cdots D_{l_d,m_d}\big)D_{i_{s+1},n_{s+1}}\cdots D_{i_r,n_r}\\
&\hspace{0.5cm}-D_{i_1,n_1}\cdots D_{i_{s-1},n_{s-1}}\ad(b_0)^{n_s}\big(D_{l_1,m_1}\cdots D_{l_d,m_d}D_{i_s+k_1-l_1+\cdots+k_d-l_d,0}\big)D_{i_{s+1},n_{s+1}}\cdots D_{i_r,n_r}\Big)_{d+r} \\
&\hspace{13.64cm} \binom{X_1,\ldots,X_{d+r}}{Y_1,\ldots,Y_{d+r}}\\
&\overset{\eqref{eq:def_Liebracket_lq}}{=} \bimap\big(\der{f}(g)\big)_{d+r}\binom{X_1,\ldots,X_{d+r}}{Y_1,\ldots,Y_{d+r}}. \qedhere
\end{align*}
\end{proof}
In particular, Theorem \ref{thm:lq_BARI_Lieiso} together with Proposition \ref{prop:BARI_Liealg} provides an alternative proof  that $(\lqq,\ariemp)$ forms a Lie algebra, cf Section \ref{sec:lq_Lie_alg_proof}.

\vspace{0.3cm}
For a bimould $A$, we define $\delta(A)$ by
\begin{align*}
\delta(A)_d\binom{X_1,\ldots,X_d}{Y_1,\ldots,Y_d}=\big(X_1Y_1+\cdots+X_dY_d\big)A_d\binom{X_1,\ldots,X_d}{Y_1,\ldots,Y_d}.
\end{align*}
Evidently, we have for any $f\in \QDbi$ that
\begin{align*}
\delta(\bimap(f))=\bimap(\delta(f)),
\end{align*}
where $\delta$ on $\QDbi$ is given in Definition \ref{def:deriv}. Therefore, a direct consequence of Theorem \ref{thm:lq_deriv} and Theorem \ref{thm:lq_BARI_Lieiso} is the following.
\begin{prop} The tuple $(\BARI,\operatorname{ari},\delta)$ is a differential Lie algebra.
\end{prop}
This proposition can be also verified by straight-forward calculations. 
\begin{rem} The depth-graded version of the derivation given in \cite[Definition 4.10]{BI24} equals exactly $\delta$. In particular, we expect that this derivation translates to a derivation on the non depth-graded space $\bmzero$ from Remark \ref{rem:lq_grDZ_q_not_inj}.
\end{rem}
\begin{rem} There is also a formulation of the embedding $\theta:(\ls,\{-,-\})\hookrightarrow(\lqq,\ariemp)$ from Theorem \ref{thm:emb_ls_lq} given in \cite{Kü19}. Roughly, this isomorphism maps a bialternal, even mould $A=(A_d(X_1,\ldots,X_d))_{d\geq0}\in \operatorname{ARI}_{\underline{\operatorname{al}}/\underline{\operatorname{al}}}$ to the bimould $\iota(A)\in\BARI$ given by
\[\iota(A)\binom{X_1,\ldots,X_d}{Y_1,\ldots,Y_d}=A(X_1,\ldots,X_d)+\operatorname{swap}(A)(Y_1,\ldots,Y_d).\]
\end{rem}

\appendix

\section{Proof of Proposition \ref{prop:dA0_and_tau}} \label{ap:proof_prop_dA0_and_tau}

We start with (i):
\[\big(\tau\circ \derrzero{\tau(w)}\circ \tau\big)(v)= \derrzero{w}(v)+w \conczero v-\tau\big(\tau(w)\conczero\tau(v)\big).\]
Let $k_1,\ldots,k_d,s_1,\ldots,s_e\geq1$, $m_1,\ldots,m_d,t_1,\ldots,t_e\geq0$. Then, with \eqref{eq:dAr0_expl} we calculate
\begin{align}  \label{eq:dAr0}
&\tau\left(\derrzero{\tau(b_0^{m_1}b_{k_1}\cdots b_0^{m_d}b_{k_d})}\big(\tau(b_0^{t_1}b_{s_1}\cdots b_0^{t_e}b_{s_e})\big)\right) \nonumber\\
&=\tau\left(\derrzero{b_0^{k_d-1}b_{m_d+1}\cdots b_0^{k_1-1}b_{m_1+1}}(b_0^{s_e-1}b_{t_e+1}\cdots b_0^{s_1-1}b_{t_1+1})\right) \nonumber\\
&= \sum_{i=1}^e \sum_{\substack{n_1+\cdots+n_d+n=m_1+\cdots+m_d \\ l_1+\cdots+l_d+l=k_1+\cdots+k_d \\ n_s,n,l\geq0,\ l_s\geq1,\ l=0 \text{ if } i=1}} (-1)^{l+n} \prod_{s=1}^d \binom{m_s}{n_s}\binom{k_s-1}{l_s-1} \nonumber\\
&\tau\big(b_0^{s_e-1}b_{t_e+1}\cdots b_0^{s_{i+1}-1}b_{t_{i+1}+1}b_0^{s_i-1}\big(b_{t_i+n+1}b_0^{l_d-1}b_{n_d+1}\cdots b_0^{l_1-1}b_{n_1+1}b_0^l\big)b_0^{s_{i-1}-1}b_{t_{i-1}+1}\cdots b_0^{s_1-1}b_{t_1+1}\big) \nonumber\\
&=\sum_{i=1}^e \sum_{\substack{n_1+\cdots+n_d+n=m_1+\cdots+m_d \\ l_1+\cdots+l_d+l=k_1+\cdots+k_d \\ n_s,n,l\geq0,\ l_s\geq1,\ l=0 \text{ if } i=1}} (-1)^{l+n} \prod_{s=1}^d \binom{m_s}{n_s}\binom{k_s-1}{l_s-1} \nonumber \\
&\hspace{2cm}b_0^{t_1}b_{s_1}\cdots b_0^{t_{i-2}}b_{s_{i-2}}b_0^{t_{i-1}}\big(b_{s_{i-1}+l}b_0^{n_1}b_{l_1}\cdots b_0^{n_d}b_{l_d}b_0^n\big)b_0^{t_i}b_{s_i}\cdots b_0^{t_e}b_{s_e} \nonumber \\
&=\sum_{i=1}^e \sum_{\substack{n_1+\cdots+n_d+n=m_1+\cdots+m_d \\ l_1+\cdots+l_d+l=k_1+\cdots+k_d \\ n_s,n,l\geq0,\ l_s\geq1,\ n=0 \text{ if } i=e}} (-1)^{l+n} \prod_{s=1}^d \binom{m_s}{n_s}\binom{k_s-1}{l_s-1} \nonumber \\
&\hspace{2cm}b_0^{t_1}b_{s_1}\cdots b_0^{t_{i-1}}b_{s_{i-1}}b_0^{t_i}\big(b_{s_i+l}b_0^{n_1}b_{l_1}\cdots b_0^{n_d}b_{l_d}b_0^n\big)b_0^{t_{i+1}}b_{s_{i+1}}\cdots b_0^{t_e}b_{s_e} \nonumber \\
&+\sum_{\substack{n_1+\cdots+n_d+n=m_1+\cdots+m_d \\ n_s,n\geq0}} (-1)^n \prod_{s=1}^d \binom{m_s}{n_s}b_0^{n_1}b_{k_1}\cdots b_0^{n_d}b_{k_d}b_0^{n+t_1}b_{s_1}\cdots b_0^{t_e}b_{s_e} \nonumber\\
&- \sum_{\substack{l_1+\cdots+l_d+l=k_1+\cdots+k_d \\ l_s\geq1,\ l\geq0}} (-1)^l \prod_{s=1}^d \binom{k_s-1}{l_s-1} b_0^{t_1}b_{s_1}\cdots b_0^{t_e}b_{s_e+l}b_0^{m_1}b_{l_1}\cdots b_0^{m_d}b_{l_d} \nonumber \\
&= \derrzero{b_0^{m_1}b_{k_1}\cdots b_0^{m_d}b_{k_d}}(b_0^{t_1}b_{s_1}\cdots b_0^{t_e}b_{s_e}) + b_0^{m_1}b_{k_1}\cdots b_0^{m_d}b_{k_d} \conczero b_0^{t_1}b_{s_1}\cdots b_0^{t_e}b_{s_e} \nonumber \\
&\hspace{0.4cm}- \tau\big(\tau(b_0^{m_1}b_{k_1}\cdots b_0^{m_d}b_{k_d})\conczero \tau(b_0^{t_1}b_{s_1}\cdots b_0^{t_e}b_{s_e})\big). \nonumber
\end{align}
Here, the last equality follows from
\begin{align*}
&b_0^{m_1}b_{k_1}\cdots b_0^{m_d}b_{k_d} \conczero b_0^{t_1}b_{s_1}\cdots b_0^{t_e}b_{s_e}= \pizero\big(\sec(b_0^{m_1}b_{k_1}\cdots b_0^{m_d}b_{k_d}) \sec( b_0^{t_1}b_{s_1}\cdots b_0^{t_e}b_{s_e})\big) \\
&\hspace{2cm}=\sec(b_0^{m_1}b_{k_1}\cdots b_0^{m_d}b_{k_d}) \pizero(\sec( b_0^{t_1}b_{s_1}\cdots b_0^{t_e}b_{s_e})) \\
&\hspace{2cm}=\sum_{\substack{n_1+\cdots+n_d+n=m_1+\cdots+m_d \\ n_s,n\geq0}} (-1)^n \prod_{s=1}^d \binom{m_s}{n_s}b_0^{n_1}b_{k_1}\cdots b_0^{n_d}b_{k_d}b_0^{n+t_1}b_{s_1}\cdots b_0^{t_e}b_{s_e}, 
\end{align*}
and
\begin{align*}
&\tau\big(\tau(b_0^{m_1}b_{k_1}\cdots b_0^{m_d}b_{k_d})\conczero \tau(b_0^{t_1}b_{s_1}\cdots b_0^{t_e}b_{s_e})\big) \\
& = \tau\big(b_0^{k_d-1}b_{m_d+1}\cdots b_0^{k_1-1}b_{m_1+1})\conczero b_0^{s_e-1}b_{t_e+1}\cdots b_0^{s_1-1}b_{t_1+1}) \\
&= \sum_{\substack{l_1+\cdots+l_d+l=k_1+\cdots+k_d \\ l_s\geq1,\ l\geq0}} (-1)^l \prod_{s=1}^d \binom{k_s-1}{l_s-1}\tau(b_0^{l_d-1}b_{m_d+1}\cdots b_0^{l_1-1}b_{m_1+1}b_0^{l+s_e-1}b_{t_e+1}\cdots b_0^{s_1-1}b_{t_1+1}) \\
&=\sum_{\substack{l_1+\cdots+l_d+l=k_1+\cdots+k_d \\ l_s\geq1,\ l\geq0}} (-1)^l \prod_{s=1}^d \binom{k_s-1}{l_s-1} b_0^{t_1}b_{s_1}\cdots b_0^{t_e}b_{s_e+l}b_0^{m_1} b_{l_1}\cdots b_0^{m_d}b_{l_d}.
\end{align*}

Next, we prove (ii):
\[\tau\circ \derlzero{\tau(w)}\circ \tau=\derlzero{\rho(w)}.\] 
For $k_1,\ldots,k_d,s_1,\ldots,s_e\geq1$, $m_1,\ldots,m_d,t_1,\ldots,t_e\geq0$, we compute with \eqref{eq:dAl0_expl}
\begin{align}  \label{eq:dAl0_tau}
&\tau\left(\derlzero{\tau(b_0^{m_1}b_{k_1}\cdots b_0^{m_d}b_{k_d})}\big(\tau(b_0^{t_1}b_{s_1}\cdots b_0^{t_e}b_{s_e})\big)\right) \nonumber\\
&=\tau\left(\derlzero{b_0^{k_d-1}b_{m_d+1}\cdots b_0^{k_1-1}b_{m_1+1}}(b_0^{s_e-1}b_{t_e+1}\cdots b_0^{s_1-1}b_{t_1+1})\right) \nonumber\\
&= \sum_{i=1}^e \sum_{\substack{n_1+\cdots+n_d+n=m_1+\cdots+m_d \\ l_1+\cdots+l_d+l=k_1+\cdots+k_d \\ n_s,n,l\geq0,\ l_s\geq1}} (-1)^{l+n} \prod_{s=1}^d \binom{m_s}{n_s}\binom{k_s-1}{l_s-1} \nonumber\\
&\tau\big(b_0^{s_e-1}b_{t_e+1}\cdots b_0^{s_{i+1}-1}b_{t_{i+1}+1}b_0^{s_i-1}\big(b_0^{l_d-1}b_{n_d+1}\cdots b_0^{l_1-1}b_{n_1+1}b_0^lb_{t_i+n+1}\big)b_0^{s_{i-1}-1}b_{t_{i-1}+1}\cdots b_0^{s_1-1}b_{t_1+1}\big) \nonumber\\
&=\sum_{i=1}^e \sum_{\substack{n_1+\cdots+n_d+n=m_1+\cdots+m_d \\ l_1+\cdots+l_d+l=k_1+\cdots+k_d \\ n_s,n,l\geq0,\ l_s\geq1}} (-1)^{l+n} \prod_{s=1}^d \binom{m_s}{n_s}\binom{k_s-1}{l_s-1}  \\
&\hspace{2cm}b_0^{t_1}b_{s_1}\cdots b_0^{t_{i-1}}b_{s_{i-1}}b_0^{t_i}\big(b_0^nb_{l+1}b_0^{n_1}b_{l_1}\cdots b_0^{n_d}b_{s_i+l_d-1}\big)b_0^{t_{i+1}}b_{s_{i+1}}\cdots b_0^{t_e}b_{s_e}. \nonumber
\end{align}
On the other hand, we have again with \eqref{eq:dAl0_expl}
\begin{align} \label{eq:dAl0_rho}
&\derlzero{\rho(b_0^{m_1}b_{k_1}\cdots b_0^{m_d}b_{k_d})}(b_0^{t_1}b_{s_1}\cdots b_0^{t_e}b_{s_e}) \nonumber\\
&=\sum_{\substack{l_1+\cdots+l_d=k_1+\cdots+k_d \\ n_1+\cdots+n_d=m_1+\cdots+m_d \\ l_s\geq1,\ n_s\geq0}} (-1)^{l_d+n_d-1}\prod_{s=1}^{d-1} \binom{k_s-1}{l_s-1} \binom{m_s}{n_s} \derlzero{b_0^{n_d}b_{l_d}b_0^{n_1}b_{l_1}\cdots b_0^{n_{d-1}}b_{l_{d-1}}}(b_0^{t_1}b_{s_1}\cdots b_0^{t_e}b_{s_e}) \nonumber\\
&= \sum_{i=1}^e \sum_{\substack{\overline{l}_1+\cdots+\overline{l}_d+\overline{l}=l_1+\cdots+l_d=k_1+\cdots+k_d \\ \overline{n}_1+\cdots+\overline{n}_d+\overline{n}=m_1+\cdots+m_d \\ l_s,\overline{l}_s\geq1,\ n_s,\overline{n}_s,\overline{n},\overline{l}\geq0}} (-1)^{\overline{l}+\overline{n}+l_d+n_d-1} \prod_{s=1}^{d-1}  \binom{k_s-1}{l_s-1} \binom{l_s-1}{\overline{l}_s-1} \binom{m_s}{n_s} \binom{n_s}{\overline{n}_s} \\
& \binom{l_d-1}{\overline{l}_d-1}\binom{n_d}{\overline{n}_d} b_0^{t_1}b_{s_1}\cdots b_0^{t_{i-1}}b_{s_{i-1}}b_0^{t_i}\big(b_0^{\overline{n}_d}b_{\overline{l}_d}b_0^{\overline{n}_1}b_{\overline{l}_1}\cdots b_0^{\overline{n}_{d-1}}b_{\overline{l}_{d-1}}b_0^{\overline{n}}b_{s_i+\overline{l}}\big) b_0^{t_{i+1}}b_{s_{i+1}}\cdots b_0^{t_e}b_{s_e}. \nonumber
\end{align}
Let $\lambda_1,\cdots,\lambda_{d-1},\lambda,\overline{\lambda} \geq 1$, $\lambda_1+\cdots+\lambda_{d-1}+\lambda+\overline{\lambda}=k_1+\cdots+k_d$, and $\mu_1,\ldots,\mu_{d-1},\mu,\overline{\mu}\geq0$, $\mu_1+\cdots+\mu_{d-1}+\mu+\overline{\mu}=m_1+\cdots+m_d$. Then, the coefficient of the word
\begin{align*}
b_0^{t_1}b_{s_1}\cdots b_0^{t_{i-1}}b_{s_{i-1}}b_0^{t_i}\big(b_0^{\mu}b_{\lambda}b_0^{\mu_1}b_{\lambda_1}\cdots b_0^{\mu_{d-1}}b_{\lambda_{d-1}}b_0^{\overline{\mu}}b_{s_i+\overline{\lambda}}\big)b_0^{t_{i+1}}b_{s_{i+1}}\cdots b_0^{t_e}b_{s_e}
\end{align*}
in \eqref{eq:dAl0_tau} is given by
\begin{align} \label{eq:dAl0_coeff1}
(-1)^{\lambda+\mu-1} \binom{m_d}{\overline{\mu}}\binom{k_d-1}{\overline{\lambda}}\prod_{s=1}^{d-1} \binom{m_s}{\mu_s}\binom{k_s-1}{\lambda_s-1},
\end{align}
and in \eqref{eq:dAl0_rho} the coefficient is given by
\begin{align} \label{eq:dAl0_coeff2}
&(-1)^{\overline{\lambda}+\overline{\mu}-1} \sum_{\substack{l_1+\cdots+l_d=k_1+\cdots+k_d \\ n_1+\cdots+n_d=m_1+\cdots+m_d \\ l_s\geq1,\ n_s\geq0}} (-1)^{l_d+n_d} \binom{l_d-1}{\lambda-1}\binom{n_d}{\mu}\prod_{s=1}^{d-1} \binom{k_s-1}{l_s-1} \binom{l_s-1}{\lambda_s-1} \binom{m_s}{n_s}\binom{n_s}{\mu_s} \nonumber\\
&=(-1)^{\overline{\lambda}+\overline{\mu}-1} \sum_{\substack{l_1+\cdots+l_d=k_1+\cdots+k_d \\ n_1+\cdots+n_d=m_1+\cdots+m_d \\ l_s\geq1,\ n_s\geq0}} (-1)^{l_d+n_d}\binom{l_d-1}{\lambda-1}\binom{n_d}{\mu}\\
&\hspace{8,8cm}\prod_{s=1}^{d-1} \binom{k_s-1}{\lambda_s-1} \binom{k_s-\lambda_s}{l_s-\lambda_s} \binom{m_s}{\mu_s}\binom{m_s-\mu_s}{n_s-\mu_s}. \nonumber
\end{align}
To prove (ii) we have to verify that the values in \eqref{eq:dAl0_coeff1} and \eqref{eq:dAl0_coeff2} agree, i.e., we have to show that
\begin{align} \label{eq:dAl0_coeff1_equal_coeff2}
&(-1)^{\lambda+\mu} \binom{m_d}{\overline{\mu}}\binom{k_d-1}{\overline{\lambda}} \\
&=(-1)^{\overline{\lambda}+\overline{\mu}} \sum_{\substack{l_1+\cdots+l_d=k_1+\cdots+k_d \\ n_1+\cdots+n_d=m_1+\cdots+m_d \\ l_s\geq1,\ n_s\geq0}} (-1)^{l_d+n_d} \binom{l_d-1}{\lambda-1}\binom{n_d}{\mu} \prod_{s=1}^{d-1}  \binom{k_s-\lambda_s}{l_s-\lambda_s} \binom{m_s-\mu_s}{n_s-\mu_s}.
\end{align}
We prove this equality in two steps. First we consider the terms of the equations containing $l_1,\ldots,l_d,\lambda_1,\ldots,\lambda_{d-1},\lambda,\overline{\lambda}$, and then the parts with $n_1,\ldots,n_d,\mu_1,\ldots,\mu_{d-1},\mu,\overline{\mu}$. First, we have
\begin{align} \label{eq:dAl0_coeff_calc_kl}
&(-1)^{\overline{\lambda}} \sum_{\substack{l_1+\cdots+l_d=k_1+\cdots+k_d \\ l_s\geq 1}} (-1)^{l_d}\binom{l_d-1}{\lambda-1} \prod_{s=1}^{d-1} \binom{k_s-\lambda_s}{l_s-\lambda_s} \\
&=(-1)^{\overline{\lambda}+k_1+\cdots+k_d+\lambda_1+\cdots+\lambda_{d-1}} \sum_{l_1=0}^{k_1-\lambda_1}\cdots \sum_{l_{d-1}=0}^{k_{d-1}-\lambda_{d-1}} \nonumber \\
&\hspace{3.8cm}\binom{k_1+\cdots+k_d-l_1-\lambda_1-\cdots-l_{d-1}-\lambda_{d-1}-1}{\lambda-1}\prod_{s=1}^{d-1} (-1)^{l_s} \binom{k_s-\lambda_s}{l_s} \nonumber \\
&=(-1)^{\lambda} \sum_{l_1=0}^{k_1-\lambda_1}\cdots \sum_{l_{d-1}=0}^{k_{d-1}-\lambda_{d-1}} \binom{k_1+\cdots+k_d-l_1-\lambda_1-\cdots-l_{d-1}-\lambda_{d-1}-1}{\overline{\lambda}-l_1-\cdots-l_{d-1}} \nonumber\\
&\hspace{12.7cm}\prod_{s=1}^{d-1} (-1)^{l_s} \binom{k_s-\lambda_s}{l_s} \nonumber\\
&=(-1)^{\lambda}\binom{k_d-1}{\overline{\lambda}}. \nonumber
\end{align}
where we used in the third step that $\lambda_1+\cdots+\lambda_{d-1}+\lambda+\overline{\lambda}=k_1+\cdots+k_d$, and in the last step \cite[Lemma 3.23]{Bu23}. Similarly, we compute
\begin{align} \label{eq:dAl0_coeff_calc_mn}
&(-1)^{\overline{\mu}} \sum_{\substack{n_1+\cdots+n_d=m_1+\cdots+m_d \\ n_s\geq0}} (-1)^{n_d} \binom{n_d}{\mu} \prod_{s=1}^{d-1} \binom{m_s-\mu_s}{n_s-\mu_s} \\
&=(-1)^{\overline{\mu}+m_1+\cdots+m_d+\mu_1+\cdots+\mu_{d-1}} \sum_{n_1=0}^{m_1-\mu_1}\cdots \sum_{n_{d-1}=0}^{m_{d-1}-\mu_{d-1}} \nonumber\\
&\hspace{4cm}\binom{m_1+\cdots+m_d-n_1-\mu_1-\cdots-n_{d-1}-\mu_{d-1}}{\mu}\prod_{s=1}^{d-1}(-1)^{n_s} \binom{m_s-\mu_s}{n_s} \nonumber\\
&=(-1)^{\mu}\sum_{n_1=0}^{m_1-\mu_1}\cdots \sum_{n_{d-1}=0}^{m_{d-1}-\mu_{d-1}} \binom{m_1+\cdots+m_d-n_1-\mu_1-\cdots-n_{d-1}-\mu_{d-1}}{\overline{\mu}-n_1-\cdots-n_{d-1}} \nonumber\\
&\hspace{12.6cm}\prod_{s=1}^{d-1}(-1)^{n_s} \binom{m_s-\mu_s}{n_s} \nonumber\\
&=(-1)^{\mu} \binom{m_d}{\overline{\mu}}. \nonumber
\end{align}
The calculations in \eqref{eq:dAl0_coeff_calc_kl} and \eqref{eq:dAl0_coeff_calc_mn} together show the claimed identity \eqref{eq:dAl0_coeff1_equal_coeff2}.

\section{List of maps}

For the reader's convenience, we list all maps in this work and their first occurrence.

\renewcommand{\arraystretch}{1.7}
\begin{table}[H]
\begin{tabular}{p{4.4cm}p{9.5cm}c}
\textbf{Section \ref{sec:balanced_qmzv}.} & & \\ \hline
$\wt:\QB\to \mathbb{Z}$ & $b_{s_1}\cdots b_{s_l}\mapsto s_1+\cdots+s_l+\#\{i\mid s_i=0\}$ & eq. \eqref{eq:wt_dep_QB}\\
$\dep:\QB\to\mathbb{Z}$ & $b_{s_1}\cdots b_{s_l}\mapsto l-\#\{i\mid s_i=0\}$ & eq. \eqref{eq:wt_dep_QB}\\
$\ast_b:\QB^{\otimes2}\to\QB$ & $	b_iv\ast_b b_jw=b_i(v\ast_b b_jw)+b_j(b_iv\ast_b w)+\delta_{ij>0} b_{i+j}(v\ast_b w)$ & Def. \ref{def:bal_quasish}\\
$\tau:\QB^0\to\QB^0$ & $b_0^{m_1}b_{k_1}\cdots b_0^{m_d}b_{k_d} \mapsto b_0^{k_d-1}b_{m_d+1}\cdots b_0^{k_1-1}b_{m_1+1}$ & Def. \ref{def:tau} 
\end{tabular} 
\end{table}
\begin{table}[H]
\begin{tabular}{p{4.4cm}p{9.5cm}c}
\textbf{Section \ref{sec:def_lq}.} & & \\ \hline
$\grD:\QB\to \QB$ & $\sum\limits_{l=0}^d \phi_l\mapsto\phi_d$ & eq. \eqref{eq:def_grD} \\
$\shuffle:\QB^{\otimes2}\to\QB$ & $b_iv\shuffle b_j w= b_i(v\shuffle b_j w)+b_j(b_iv\shuffle w)$ & eq. \eqref{eq:shuffle}\\
$\shco:\QB\to \QB^{\otimes2}$ & $b_i\mapsto b_i\otimes 1+1\otimes b_i$ & eq. \eqref{eq:shco_B} 
\end{tabular} 
\end{table}
\begin{table}[H]
\begin{tabular}{p{4.4cm}p{9.5cm}c}
$\pizero:\QB\to \QB^0$ & $ wb_0\mapsto 0$ & Def. \ref{def:lq} \\
\end{tabular} 
\end{table}
\begin{table}[H]
\begin{tabular}{p{4.4cm}p{9.5cm}c}
\textbf{Section \ref{sec:lq_Lie_alg}.} & & \\ \hline
$\der{w}:\QB\to \QB$ & $b_i\mapsto\hspace{-0.3cm}\sum\limits_{\substack{l_1+\cdots+l_d+l \\=k_1+\cdots+k_d}}(-1)^l\prod\limits_{s=1}^d\binom{k_s-1}{l_s-1}[b_{i+l},b_0^{m_1}b_{l_1}\cdots b_0^{m_d}b_{l_d}b_0^{m_{d+1}}]$ & Def. \ref{def:Lie_bracket_lq} \\
& for $w=b_0^{m_1}b_{k_1}\cdots b_0^{m_d}b_{k_d}b_0^{m_{d+1}}$ & \\
$\ari{\cdot}{\cdot}:\QB^{\otimes2}\to\QB$ & $\ari{f}{g}= \der{f}(g)-\der{g}(f)+[f,g]$ & Def \ref{def:Lie_bracket_lq} \\
\end{tabular} 
\end{table}
\begin{table}[H]
\begin{tabular}{p{4.4cm}p{9.5cm}c}
\textbf{Section \ref{sec:lq_Lie_alg_proof}.} & & \\ \hline
$\sec:\QB^0\to\QB$ & $b_0^{m_1}b_{k_1}\cdots b_0^{m_d}b_{k_d}\mapsto\sum\limits_{\substack{n_1+\cdots+n_d+n \\ =m_1+\cdots+m_d}} (-1)^n\prod\limits_{s=1}^d\binom{m_s}{n_s}$ & eq. \eqref{eq:sec_expl} \\
& $\hspace{6.5cm}b_0^{n_1}b_{k_1}\cdots b_0^{n_d}b_{k_d}b_0^n$ & \\
$\partial_0:\QB\to\QB$ & $b_0\mapsto 1,\ b_i\mapsto0$, $i\geq1$ & eq. \eqref{eq:def_d0} \\
$\rho:\QB^0\to\QB^0$ & $b_0^{m_1}b_{k_1}\cdots b_0^{m_d}b_{k_d}\mapsto\hspace{-0.2cm}\sum\limits_{\substack{l_1+\cdots+l_d=k_1+\cdots+k_d \\ n_1+\cdots+n_d=m_1+\cdots+m_d}} \hspace{-0.5cm} (-1)^{l_d+n_d-1}$ & Def. \ref{def:rho} \\ & $\hspace{3cm}\prod\limits_{s=1}^{d-1} \binom{k_s-1}{l_s-1}\binom{m_s}{n_s}b_0^{n_d}b_{l_d}b_0^{n_1}b_{l_1}\cdots b_0^{n_{d-1}}b_{l_{d-1}}$ & \\
$S:\QB\to \QB$ & $b_{s_1}\cdots b_{s_l}\mapsto (-1)^l b_{s_l}\cdots b_{s_1}$ & eq. \eqref{eq:antipode} \\
$S_0:\QB^0\to\QB^0$ & $S_0=\pizero\circ S\circ \sec$ & eq. \eqref{eq:S_0} \\
$\derr{w}:\QB\to\QB$ & $b_i\mapsto\sum\limits_{\substack{l_1+\cdots+l_d+l \\ =k_1+\cdots+k_d}}(-1)^l\prod\limits_{s=1}^d \binom{k_s-1}{l_s-1} b_{i+l}b_0^{m_1}b_{l_1}\dots b_0^{m_d}b_{l_d}b_0^{m_{d+1}}$ & eq. \eqref{eq:derr} \\
& for $w=b_0^{m_1}b_{k_1}\dots b_0^{m_d}b_{k_d}b_0^{m_{d+1}}$ & \\
$\derl{w}:\QB\to\QB$ & $b_i\mapsto\sum\limits_{\substack{l_1+\cdots+l_d+l \\ =k_1+\cdots+k_d}}(-1)^l\prod\limits_{s=1}^d \binom{k_s-1}{l_s-1} b_0^{m_1}b_{l_1}\dots b_0^{m_d}b_{l_d}b_0^{m_{d+1}}b_{i+l}$ & eq. \eqref{eq:derl} \\
& for $w=b_0^{m_1}b_{k_1}\dots b_0^{m_d}b_{k_d}b_0^{m_{d+1}}$ & \\
$\derrzero{w}:\QB^0\to\QB^0$ & $\derrzero{w}=\pizero\circ \derr{\sec(w)}\circ\sec$ & eq. \eqref{eq:dAr0_expl} \\
$\derlzero{w}:\QB^0\to\QB^0$ & $\derlzero{w}=\pizero\circ \derl{\sec(w)}\circ\sec$ & eq. \eqref{eq:dAl0_expl} \\
$\derzero{w}:\QB^0\to\QB^0$ & $\derzero{w}=\pizero\circ \der{\sec(w)}\circ\sec$ & eq. \eqref{eq:derzero} 
\end{tabular} 
\end{table}
\begin{table}[H]
\begin{tabular}{p{4.4cm}p{9.5cm}c}
$\conczero:(\QB^0)^{\otimes2}\to\QB^0$ & $v \conczero w=\pizero\big(\sec(v),\sec(w)\big) $ & eq. \eqref{eq:conzero} \\
$\arizero{\cdot}{\cdot}:(\QB^0)^{\otimes2}\to$ & $\arizero{v}{w}=\pizero\big(\ari{\sec(v)}{\sec(w)}\big)$ & eq. \eqref{eq:arizero} \\
$\hspace{3cm}\QB^0$ & &
\end{tabular} 
\end{table}
\begin{table}[H]
\begin{tabular}{p{4.4cm}p{9.5cm}c}
\textbf{Section \ref{sec:deriv_on_lq}.} & & \\ \hline
$\tau_{\Dbi}:\QDbi\to\QDbi$ & $\tau_{\Dbi}=\sec\circ\tau\circ\pizero$ & eq. \eqref{eq:def_tau_Dbi} \\
$\delta:\QDbi\to\QDbi$ & $D_{k,m}\mapsto D_{k+1,m+1}$ & Def. \ref{def:deriv}
\end{tabular} 
\end{table}
\begin{table}[H]
\begin{tabular}{p{4.4cm}p{9.5cm}c}
\textbf{Section \ref{sec:ls_and_lq}.} & & \\ \hline
$\shco:\QX\to \QX^{\otimes2}$ & $x_i\mapsto x_i\otimes 1 +1 \otimes x_i$ & eq. \eqref{eq:shco_QX} \\
$\shco:\QY\to \QY^{\otimes2}$ & $y_i\mapsto y_i\otimes 1 +1 \otimes y_i$ & eq. \eqref{eq:shco_QY} \\
$\piY:\QX\to\QY$ & $x_0^{k_1-}x_1\cdots x_0^{k_d-1}x_1\mapsto y_{k_1}\cdots y_{k_d}$ & eq. \eqref{eq:piY} \\
$d_w:\QX\to\QX$ & $x_0\mapsto 0,\ x_1\mapsto [x_1,w]$ & Def. \ref{def:Ihara_bracket} \\
$\{\cdot,\cdot\}:\QX^{\otimes2}\to\QX$ & $\{f,g\}=d_f(g)-d_g(f)+[f,g]$ & Def. \ref{def:Ihara_bracket}\\
$\thetaX: \QX\to \QB$ & $x_{s_1}\cdots x_{s_k} \mapsto b_{s_1}\cdots b_{s_k}$ & eq. \eqref{eq:thetaX} \\
$\thetaY:\QY\hookrightarrow \QB$ & $y_{k_1}\cdots y_{k_d} \mapsto b_{k_d}\cdots b_{k_1}$ & eq. \eqref{eq:thetaY} \\
$\theta:\ls\to\lqq$ & $\phi\mapsto \thetaX(\phi)+\thetaY(\piY(\phi))$ & Thm. \ref{thm:emb_ls_lq} 
\end{tabular} 
\end{table}
\begin{table}[H]
\begin{tabular}{p{4.4cm}p{9.5cm}c}
\textbf{Section \ref{sec:bimoulds}.} & & \\ \hline
$\bimap_d:\QDbi_d\to$ & $D_{k_1,m_1}\cdots D_{k_d,m_d}\mapsto X_1^{k_1-1}Y_1^{m_1}\cdots X_d^{k_d-1}Y_d^{m_d}$ & Def. \ref{def:bimap} \\
$\QQ[X_1,Y_1,\ldots,X_d,Y_d]$ & &
\end{tabular} 
\end{table}

\end{document}